\documentclass[10pt]{amsart}

\usepackage[square,compress,comma, numbers,sort]{natbib}
\usepackage[colorlinks=true, citecolor=red, linkcolor=blue]{hyperref}
\usepackage{amsfonts,mathtools}

\allowdisplaybreaks[4]

\usepackage{amsmath}
\usepackage{amssymb}
\usepackage{mathtools}

\usepackage{mathrsfs}

\usepackage{color}

\DeclareMathOperator{\cov}{\mathbb{C}ov}
\DeclareMathOperator{\var}{\mathbb{V}ar}

\def\argmin {{\rm argmin}}

\definecolor{c20}{rgb}{0.,0.7,0.}
\definecolor{c30}{rgb}{0.,0.,1.}
\definecolor{c40}{rgb}{1,0.1,0.7}
\definecolor{c50}{rgb}{1,0,0}
\definecolor{c60}{rgb}{1,0.9,0.1}

\def\x{\vk{x}}

\newcommand{\abs}[1]{\left\lvert #1 \right\rvert}

\newcommand{\sprod}[1]{\langle#1\rangle}

\newcommand{\E}[1]{\mathbb{E}\left\{ #1\right\}}

\newcommand{\sgn}[1]{{\rm sgn}\left( #1\right)}

\newcommand{\pk}[1]{\mathbb{P} \left\{ #1 \right \} }

\newcommand{\R}{\mathbb{R}}
\newcommand{\N}{\mathbb{N}}

\newcommand{\ldot}{,\ldots,}
\newcommand{\ep}{\varepsilon}
\renewcommand{\tilde}{\widetilde}
\renewcommand{\hat}{\widehat}

\newcommand{\limit}[1]{\lim_{#1 \to \infty}}

\newcommand{\BQN}{\begin{eqnarray}}
\newcommand{\EQN}{\end{eqnarray}}
\newcommand{\BQNY}{\begin{eqnarray*}}
\newcommand{\EQNY}{\end{eqnarray*}}

\newcommand{\BS}{\begin{sat}}
\newcommand{\ES}{\end{sat}}
\newcommand{\BT}{\begin{theo}}
\newcommand{\ET}{\end{theo}}
\newcommand{\BK}{\begin{korr}}
\newcommand{\EK}{\end{korr}}

\newcommand{\BD}{\begin{de}}
\newcommand{\ED}{\end{de}}

\newcommand{\BIT}{\begin{itemize}}
\newcommand{\EIT}{\end{itemize}}
\newcommand{\BDI}{\begin{description}}
\newcommand{\EDI}{\end{description}}

\newcommand{\BRM}{\begin{remark}}
\newcommand{\ERM}{\end{remark}}

\newcommand{\BEL}{\begin{lem}}
\newcommand{\EEL}{\end{lem}}

\newtheorem{theo}{Theorem}[section]
\newtheorem{sat}[theo]{Proposition}
\newtheorem{de}[theo]{Definition}
\newtheorem{lem}[theo]{Lemma}

\newtheorem{example}[theo]{Example}
\newtheorem{korr}[theo]{Corollary}
\newtheorem{remark}[theo]{Remark}

\newcommand{\nelem}[1]{{Lemma \ref{#1}}}

\newcommand{\COM}[1]{}

\def\td{\text{\rm d}}
\def\td{d}
\def\IF{\infty}

\newcommand{\kb}[1]{\boldsymbol{#1}}
\newcommand{\vk}[1]{\kb{#1}}

\def\IF{\infty}

\def\bqny#1{\begin{eqnarray*} #1 \end{eqnarray*}}
\def\bqn#1{\begin{eqnarray} #1 \end{eqnarray}}

\begin{document}

\title{Simultaneous ruin probability for multivariate Gaussian risk model }

\author{Krzysztof Bisewski}
\address{Krzysztof Bisewski, Department of Actuarial Science, 
	University of Lausanne,\\
	UNIL-Dorigny, 1015 Lausanne, Switzerland
}
\email{Krzysztof.Bisewski@unil.ch}

\author{Krzysztof D\c{e}bicki}
\address{Krzysztof D\c{e}bicki, Mathematical Institute, University of Wroc\l aw, pl. Grunwaldzki 2/4, 50-384 Wroc\l aw, Poland}
\email{Krzysztof.Debicki@math.uni.wroc.pl}

\author{Nikolai Kriukov}
\address{Nikolai Kriukov, Department of Actuarial Science, 
University of Lausanne,\\
UNIL-Dorigny, 1015 Lausanne, Switzerland
}
\email{Nikolai.Kriukov@unil.ch}

\bigskip

\date{\today}
\maketitle

{\bf Abstract:}
Let
$\vk Z(t)=(Z_1(t) \ldot Z_d(t))^\top , t \in \R$
where $Z_i(t), t\in \R$, $i=1,...,d$ are
mutually independent centered Gaussian
processes with continuous sample paths a.s. and stationary increments.
For $\vk X(t)= A \vk Z(t),\ t\in \R$, where
$A$ is a nonsingular $d\times d$ real-valued matrix,
$\vk u, \vk c\in\R^d$ and $T>0$
we derive tight bounds for
\[
\pk{\exists_{t\in [0,T]}: \cap_{i=1}^d \{ X_i(t)- c_i t > u_i\}}
\]
and find exact asymptotics
as $(u_1,...,u_d)^{\top}= (u a_1,..., ua_d)^\top$ and $u\to\infty$.



\section{Introduction}
Let
$\vk Z(t)=(Z_1(t) \ldot Z_d(t))^\top , t \in \R$ be a $d$-dimensional
Gaussian process, where  $Z_i(t), t\in \R$, $i=1,...,d$ are mutually independent centered Gaussian
processes with continuous sample paths a.s. and stationary increments.
For $\vk u, \vk c\in\R^d$ and $T>0$ we
consider
\begin{eqnarray*}
p_T(\vk u)&:=&
\pk{\exists_{t\in [0,T]}: \vk X(t)- \vk c t > \vk u} \\
&=&
\pk{\exists_{t\in [0,T]}: \cap_{i=1}^d \{ X_i(t)- c_i t > u_i\}},
\end{eqnarray*}
where
$\vk X(t)= A \vk Z(t),$ 
with $A$ a nonsingular $d\times d$ real-valued matrix.

In the main result of this contribution, which is
Theorem \ref{th.exact}, we derive exact asymptotics of $p_T(\vk u)$ for $\vk u =u \vk a=(a_1 u,...,a_d u)^\top$, as $u\to\infty$,
where $\vk a \in \R^d\setminus(-\infty,0]^d$.
The core assumption that we work with is the so-called {\it Berman condition}
\[v_i(t) := \var(Z_i(t)) = o(t),\ {\rm as} \ t\to0\]
for $i=1,...,d$.
Interestingly, while in the 1-dimensional case, under Berman condition,
\[p_T( u)\sim \pk{  X(T)-  c T >  u}, \ {\rm as}\  u\to\infty,\]
(see
the seminal paper by Berman \cite{MR789369} and \cite{DeR02} for the non-centered case),
the vector-valued case considered in this manuscript leads to more
diverse scenarios that can be captured in the form
\[
p_T( u\vk a)\sim \mathcal C \cdot \pk{ \vk X(T)-  \vk c T >  u\vk a},
\]
as $u\to\infty$, where $\mathcal{C}\ge 1$ is a constant depending on the model parameters; see Eq.~(\ref{def:C}) below.

We note that, for
$\vk Z$ being a two-dimensional standard Brownian motion,
the asymptotic behavior of $p_T( u\vk a)$
was recently analyzed in \cite{DHM19}, where the strategy of the proof
was based on the independence of increments and self-similarity of Brownian motion. In general, Gaussian processes with stationary increments do not have these properties and thus the proof of the main result of this contribution
needs more subtle and refined analysis than the one used in \cite{DHM19}.
More precisely, the idea of the proof of Theorem \ref{th.exact} is based on two steps:
(i) showing that
\[
\lim_{L\to\infty}\lim_{u\to\infty}
\frac{p_T( u\vk a)}
{\pk{\exists_{t\in [T-Lu^{-2},T]}: \cap_{i=1}^d \{ X_i(t)- c_i t > u_i\}}}=1,
\]
and (ii) finding the exact asymptotics of the denominator
above.
In the first step, particularly precise analysis is needed
for the neighbourhood of the right end of the parameter set $[0,T-Lu^{-2}]$, see Lemma \ref{lem:gap}.

Complemantary to the exact asymptotics derived in Theorem \ref{th.exact},
in Theorem \ref{th.bound} we establish uniform upper and lower bounds for $p_T(\vk u)$.
This result extends recently derived bounds for $Z$ being a $d$-dimensional standard
Brownian motion \cite{KoW20, DHM19, DHK21}.

The quantity $p_T(\vk u)$ has a natural interpretation as
the {\it simultaneous ruin probability} in time horizon $[0,T]$
of an insurance portfolio represented by $d$ mutually dependent risk processes
$(R_i(t),...,R_d(t))^\top =\vk R(t,\vk u)$, where
$\vk R(t,\vk u)=\vk u-\vk X (t)+\vk c t,\  t\in\R$,
since
\[
p_T(\vk u)=
\pk{\exists_{ t\in [0,T]}: \vk R(t,\vk u)< \vk 0},
\]
where, for the $i$-th business line, $u_i$ is the initial capital, $X_i(t)$ is the accumulated clam size in time interval $[0,t]$
and $c_i$ is the premium rate.
In this context, our results complement  work
\cite{Ji18}, where the particular case $d=2, T=\infty$ with $X_2(t)=\sigma_2 X_1(t)$
where $X_1$ is a fractional Brownian motion,
was analyzed.
We refer to, e.g., \cite{FPP08,HuJ13,FPR17,Mic20} for recent works on
simultaneous ruin probability for L{\'e}vy processes and to
recently derived asymptotics for centered vector valued Gaussian processes;
see \cite{DHW20,HHJT15}.

Our findings cover two special cases that play important role in the literature on the Gaussian risk models,
i.e. fractional Brownian motion risk model and Gaussian integrated risk model; see Section \ref{s.main} for details.
We refer to \cite{Mic98, HuP99, HuP04, Deb02}
for the analysis of Gaussian risk models in $d=1$ dimensional setting.


\section{Notation}
We follow the notational convention of \cite{debicki2020extremes}.
All vectors in $\R^d$ are written in bold letters, for instance
$\vk b = (b_1, \ldots, b_d)^{\top}$, $\vk{0} = (0, \ldots, 0)^{\top}$ and $\vk{1} = (1, \ldots, 1)^{\top}$.
We follow the convention that $1$-dimensional vectors are vertical.
For two vectors $\vk x$ and $\vk y$, we write $\vk{x} > \vk {y}$ if $x_i > y_i$ for all $1 \leq i \leq d$.
For any $\vk{x}, \vk{y}\in\R^d$ we use $\langle\vk{x},\vk{y}\rangle$ for scalar product and
$\vk{x}\vk{y}$ for a component-wise product.

Given a real-valued matrix $A$ we shall write $A_{IJ}$ for the
submatrix of $A$ determined by keeping the rows and columns of $A$ with row indices in the non-empty set $I$
and column indices in the non-empty set $J$, respectively.
In our notation $\mathcal{I}_d$ is the $d\times d$ identity matrix and
${\rm diag}(\vk x)={\rm diag}(x_1,\ldots,x_d)$
stands for the diagonal matrix with entries $x_i$, $i=1,\ldots,d$ on the main diagonal, respectively.

Let in the sequel $\Sigma \in \R^{d \times d}$ be any positive definite matrix.
We write $\Sigma_{IJ}^{-1}:=\left(\Sigma_{IJ}\right)^{-1}$ for the inverse matrix
of $\Sigma_{IJ}$ whenever it exists.
For  $\vk a \in \R^d \setminus (-\IF, 0]^d $,  let $\Pi_\Sigma(\vk a)$
denotes the quadratic programming problem
\begin{equation}
\label{eq:QP}
\Pi_\Sigma(\vk{a})= \text{minimise $ \vk x^\top \Sigma^{-1} \x $ under the linear constraint } \vk{x} \ge \vk{a}.
\end{equation}
It is known that $\Pi_\Sigma(\vk a)$ has a unique solution $\tilde{\vk a} \ge \vk{a}$ and there exists a unique non-empty index set $I\subset \{1 \ldot d\}$ such that
\begin{equation}
\label{keyWW}
\tilde{\vk a}_I=\vk{a}_I,  \quad \tilde{\vk a}_J = \Sigma_{IJ}\Sigma_{II}^{-1}\vk{a}_I \vk
\ge  \vk{a}_J,  \quad \vk{\lambda}_I= \Sigma_{II}^{-1}\vk{a}_I> \vk{0}_I, \quad \vk{\lambda}_J= \vk{0}_J,
\end{equation}
where $\vk{\lambda} = \Sigma ^{-1} \tilde{\vk a}$ and $J= \{ 1 \ldot d\} \setminus I$. Additionally, we define $U := \{i\in J: \tilde a_i = a_i\}$. The coordinates $J$ (which can be empty) are responsible for dimension-reduction phenomena, while coordinates belonging to $I$ play essential role in the exact asymptotics. We refer to \nelem{lem:quadratic_programming} below for more details.

Throughout the manuscript, let $\Sigma(t)$ denote the variance matrix of process $\vk{X}$ at time $t\in[0,T]$, that is
$$\Sigma(t) := \E{\vk X(t)\vk X^\top(t)} = A\E{\vk Z(t){\vk Z}^\top(t)}A^\top = A~{\rm diag}(\vk v(t))A^\top $$
where $\vk v(t)=\left(v_1(t)\ldots v_d(t)\right)^\top$. Moreover, for all $i\in\{1,\ldots,d\}$ let
\begin{equation}\label{eq:def_rho_i}
\rho_i(t,s) := \cov(Z_i(t),Z_i(s))=\frac{v_i(s)+v_i(t)-v_i(\abs{s-t})}{2},
\end{equation}
where in the second equality we used the fact that $Z_i$ has stationary increments.
For all $t\in(0,T]$, let $\tilde{\vk{a}}(t)$ be the solution of quadratic programming problem
$\Pi_{\Sigma(t)}(\vk{a})$, defined in \eqref{eq:QP} and
$D(t) := \tilde{\vk a}(t)^\top\Sigma^{-1}(t)\tilde{\vk a}(t)$. Moreover, let $\vk{\lambda}(t) := \Sigma^{-1}(t)\tilde{\vk{a}}(t)$, and $I_t := \{i : \lambda_i(t) > 0\}$, $J_t := \{1,\ldots,d\}\setminus I_t$ (which can be empty). Throughout the manuscript we slightly abuse the notation by writing $\lambda_i(t)$ instead of $\lambda(t)_i$, $a_i(t)$ instead of $a(t)_i$, and $\tilde a_i(t)$ instead of $\tilde a(t)_i$.

\section{Main Results}\label{s.main}
Consider a centered, $d-$dimensional Gaussian process with stationary increments,
 continuous sample paths and mutually independent components $\vk Z(t),t\ge0$. Let $v_i(t) := \var Z_i(t)$ be the variance function of process $Z_i$. Due to the stationarity of increments, the covariance structure of $Z_i$ is determined by its variance function $v_i$, cf. \eqref{eq:def_rho_i}. We shall assume that, for each $i\in\{1,\ldots,d\}$,

\begin{itemize}
\item[{\rm \bf B0.}]\label{B0} $v_i\in C^1([0,T])$ is strictly increasing and $v_i(0)=0$.
\item[{\rm \bf BI.}]\label{BI} The first derivative $\dot{v_i}(T)>0$.
\item[{\rm \bf BII.}]\label{BII} $v_i(t) = o(t)$, as $t\to0$.
\end{itemize}

The following families of Gaussian processes satisfy assumptions {\bf  B0-BII}:
\vspace{0.1cm}

$\diamond$ {\it fractional Brownian motions}: $\vk Z(t)=(B_{\alpha_1}(t),...,B_{\alpha_d}(t))^\top,t\ge0$,
where $B_{\alpha_i}(t),t\ge0$, $i=1,...,d$ are mutually independent standard fractional Brownian motions with
Hurst parameters $\alpha_i/2\in(1/2,1)$, that is centered Gaussian processes with stationary increments, continuous sample paths a.s.
and variance function $v_i(t)=t^{\alpha_i}$ respectively.
We refer to, e.g., \cite{Mic98,HuP99, DeR02} for the motivation and relations of this class of stochastic processes in risk theory.

$\diamond$ {\it integrated stationary processes}:
$\vk Z(t)=(Z_{1}(t),...,Z_{d}(t))^\top,t\ge0$,
where
$Z_i(t)=\int_0^t \eta_i(s)ds$, with $\eta_i(t),t\ge0$, $ i=1,...,d$ mutually independent centered stationary Gaussian processes
with continuous sample paths a.s. and continuous strictly positive covariance function $R_i(t):=\cov (Z_i(s),Z_i(s+t))$.
One can check that $v_i(t)=2\int_0^t ds \int_0^s R(w) dw $ in this case. We refer to \cite{HuP04,Deb02,DeR02}
for the analysis of extremes of this class of processes
in the context of Gaussian risk theory and its relations to Gaussian fluid gueueing models.
\\

In the following, $\mathcal N_d(\mu,\Sigma)$ stands for the law of a $d$-dimensional normal distribution with mean $\mu\in\R^d$ and covariance matrix $\Sigma\in\R^{d\times d}$.

\BT\label{th.bound}
Let $\vk X (t)= A\vk Z (t)$, $t\geq0$ be such that
$\vk Z$ satisfies {\bf B0} with $v_i(t)=v(t)$ for all $i$, $v(t)$ is convex, and $A\in\R^{d\times d}$ is a non-singular matrix satisfying $(A^\top A)_{ij}\geq 0$. Then, for each $\vk u\in \R^d$ and $\vk c\ge \vk 0$,
\[
\pk{\vk X(T)- \vk c T > \vk u}
\le
\pk{\exists_{t\in [0,T]}: \vk X(t)- \vk c t > \vk u}
\le
\frac{\pk{ \vk X(T)- \vk c T > \vk u}}
{\pk{ \mathcal N_d(0,A^\top A) > \vk 0}}.
\]
\ET

\BRM
In the case when $\vk Z$ is a standard $d$-dimensional Brownian motion
the assumption $(A^\top A)_{ij}\geq 0$ can be lifted
and the upper bound in Theorem \ref{th.bound}
holds for any non-singular matrix $A$; see \cite{DHK21}.
It can be verified that this bound holds also for all $u$ large enough for the process $\vk Z$
considered in Example~\ref{example:2dim_fbm} below, which
suggests that
the upper bound in Theorem \ref{th.bound} holds for any non-singular matrix $A$.
\ERM

To the end of this paper, let
\begin{equation}\label{def:C}
\mathcal{C} := \frac{\sum_{i=1}^{d}\max\left(\lambda_i \cdot (A QA^{-1}\tilde{\vk a})_i,0\right)}{\sum_{i=1}^{d}\lambda_i \cdot (A QA^{-1}\tilde{\vk a})_i}
\end{equation}
where $Q = {\rm diag}(\dot v_i(T)/v_i(T))$, and $\vk\lambda$, $\vk{\tilde a}$ correspond to the solution of the quadratic problem $\Pi_{\Sigma(T)}(\vk a)$ defined in \eqref{eq:QP}; see also Lemma \ref{lem:quadratic_programming} below.

\BT\label{th.exact} Let $\vk X (t)= A\vk Z (t)$, $t\geq 0$ be such that
$\vk Z $ satisfies {\bf B0-BII}, $\vk a \in \R^d \setminus (-\infty,0]^d$, $\vk c\in\R^d$ and $A$ is a non-singular matrix. Then,
\[
\pk{\exists_{t\in [0,T]}: \vk X(t)- \vk c t > u \vk a}
\sim \mathcal{C} \cdot \pk{\vk X(T)- \vk c T >  u\vk a}, \quad u\to\infty,
\]
where $\mathcal{C}$ defined in \eqref{def:C} is a positive constant.
\ET
The heuristic interpretation of the bounds and asymptotics derived in Theorems \ref{th.bound}, \ref{th.exact} is that only a small area around the end point $T$
of the parameter set $[0,T]$ contributes to the tail distribution of the analyzed problem.
We refer to \cite{Has05,Has19} and references therein
for the analysis of the exact form of the asymptotics for
$\pk{\vk X(T)- \vk c T >  u\vk a}$, as $u\to\infty$; see also Lemma \ref{prop:proj1}.


The following example illustrates the main findings of this section.
\begin{example}\label{example:2dim_fbm}
Suppose that $d=2$, and $Z_1(t),Z_2(t)$ are mutually independent and identically distributed centered Gaussian
processes that satisfy {\bf B0-BII}. Then the constant $\mathcal{C}$ has the following form
$$
\mathcal{C}= \frac{\sum_{i=1}^{2}\max\left(\lambda_i \cdot \tilde{\vk a}_i,0\right)}{\sum_{i=1}^{2}\lambda_i \cdot \tilde{\vk a}_i}.
$$
We assume further that
$A=\left(\begin{smallmatrix} 1 & 0 \\
     \rho  & \sqrt{1-\rho^2} \\
      \end{smallmatrix}\right)$, where $\rho\in(-1,1)$ and $\vk a=(1,a)^\top$, with $a\le 1$.
\\
$\diamond$ If $a<\rho$, then $I=\{1\}$, $J=\{2\}$ and hence, as $u\to\infty$,
\[
\pk{\exists_{t\in [0,T]}: \vk X(t)- \vk c t > u \vk a}
\sim
     \pk{\vk X(T)-\vk c T > u}
\sim
     \pk{ X_1(T)-  c_1 T > u};
\]
$\diamond$ If $a=\rho$, then $I=\{1\}$, $U=\{2\}$ and as $u\to\infty$,
\[
\pk{\exists_{t\in [0,T]}: \vk X(t)- \vk c t > u \vk a}
\sim
     \pk{\vk X(T)-\vk c T > u}
\sim
     \pk{X_2(T)>c_2T|X_1(T)=c_1T}\pk{X_1(T)-c_1 T > u};
\]
$\diamond$ If $a>\rho$, then $I=\{1,2\}$,  $\vk{\tilde a}=\vk a$ and
$
\vk\lambda=\Sigma^{-1}(T)\vk a=\frac{v(T)}{1-\rho^2}\left(\begin{smallmatrix}1-a\rho \\a-\rho\end{smallmatrix}\right)
$, see also \cite{DHM19}.
Thus, if further $a\geq 0$, then  $\mathcal{C}=1$
and
\[
\pk{\exists_{t\in [0,T]}: \vk X(t)- \vk c t > u \vk a}
\sim
     \pk{ X_1(T)-  c_1 T > u, X_2(T)-  c_2 T > au},
\]
as $u\to\infty$.
Otherwise, if $a<0$, then
$\mathcal{C}=\frac{1-a\rho}{1-a\rho+a^2-a\rho}=\frac{1-a\rho}{1-2a\rho+a^2}>1
$
and hence, as $u\to\infty$,
\[
\pk{\exists_{t\in [0,T]}: \vk X(t)- \vk c t > u \vk a}
\sim
     \frac{1-a\rho}{1-2a\rho+a^2}\pk{ X_1(T)-  c_1 T > u, X_2(T)-  c_2 T > au }.
\]
In the above we used \cite[Lemma~4.2]{Has19} for the exact asymptotics of $\pk{\vk X(T)-\vk c T > u}$.
\end{example}

\section{Proofs of Main Results}

\begin{proof}[Proof of Theorem \ref{th.bound}]
Using the fact that $\vk c \geq \vk 0$ and $\vk u >\vk 0$, 
we have
\bqny{
\pk{\exists_{t\in[0,T]}\vk X(t)-\vk c t>\vk u}&=&\pk{\bigcup_{t\in[0,T]}\bigcap_{i=1}^{d}\{X_i(t)>u_i+c_it\}}\\
&=&1-\pk{\bigcap_{t\in[0,T]}\bigcup_{i=1}^{d}\left\{\frac{X_i(t)}{u_i+c_it}\leq 1\right\}}\\
&=&1-\pk{\bigcap_{t\in[0,T]}\bigcup_{i=1}^{d}\left\{\frac{-X_i(t)}{u_i+c_it}\geq -1\right\}}\\
&=&1-\pk{\bigcap_{t\in[0,T]}\bigcup_{i=1}^{d}\left\{\frac{X_i(t)}{u_i+c_it}\geq -1\right\}},
}
where in the last equality above we used that $X_i$ are centered.

Let $B_i(t),t\ge0$, $i=1,2,...,d$ mutually independent standard Brownian motions,
and $\vk B^*(t)=A\vk B(t)$.

Next, we show that
\bqny{
\pk{\bigcap_{t\in[0,T]}\bigcup_{i=1}^{d}\left\{\frac{X_i(t)}{u_i+c_it}\geq -1\right\}}&\geq&
\pk{\bigcap_{t\in[0,T]}\bigcup_{i=1}^{d}\left\{\frac{B^*_i(v(t))}{u_i+c_it}\geq -1\right\}}\label{in}
}
for which, by Gordon inequality (see e.g. \cite{MR800188} or \cite[page~55]{MR1088478}),
it suffices to check that:
\bqn{
\E{X_i(t)^2}&=&\E{B^*_i(v(t))^2}\label{g1}\\
\E{X_i(t)X_j(t)}&=&\E{B^*_i(v(t))B^*_j(v(t))}\label{g2}\\
\E{X_i(t)X_j(s)}&\geq&\E{B^*_i(v(t))B^*_j(v(s))},\quad\text{for}~t\not=s.\label{g3}
}
For all $i,j\in\{1\ldot d\}$ and $t\in[0,T]$ we have
\bqny{
\E{X_i(t)X_j(t)}=\E{(AZ)_i(t)(AZ)_j(t)}=\E{\sum_{k=1}^d a_{ik}Z_k(t)\sum_{k=1}^d a_{jk}Z_k(t)}=\E{\sum_{k=1}^d a_{ik}a_{jk}Z_k^2(t)}\\
=\sum_{k=1}^d a_{ik}a_{jk}\E{Z_k^2(t)}=\sum_{k=1}^d a_{ik}a_{jk}v(t)=(AA^\top)_{i,j}v(t).
}
Analogously,
\bqny{
\E{B^*_i(v(t))B^*_j(v(t))}=\E{(AB)_i(v(t))(AB)_j(v(t))}=\E{\sum_{k=1}^d a_{ik}B_k(v(t))\sum_{k=1}^d a_{jk}B_k(v(t))}\\
=\E{\sum_{k=1}^d a_{ik}a_{jk}B_k^2(v(t))}=\sum_{k=1}^d a_{ik}a_{jk}\E{B_k^2(v(t))}=\sum_{k=1}^d a_{ik}a_{jk}v(t)=(AA^\top)_{i,j}v(t).
}
Hence, equalities (\ref{g1}), (\ref{g2}) are satisfied.
For $t\neq s$  we  obtain that
\bqny{
\E{X_i(t)X_j(s)}&=&(AA^\top)_{i,j}\E{Z_1(t)Z_1(s)}=(AA^\top)_{i,j}\frac{v(s)+v(t)-v(\abs{s-t})}{2},\\
\E{B^*_i(v(t))B^*_j(v(s))}&=&(AA^\top)_{i,j}\E{B_1(v(t))B_1(v(s))}=(AA^\top)_{i,j} \min(v(t),v(s)).
}
As $(AA^\top)_{i,j}\geq 0$, it is enough to show that
\bqny{
\frac{v(s)+v(t)-v(\abs{s-t})}{2} \geq \min(v(t),v(s)).
}
Using the convexity of $v(\cdot)$, we have for all $s<t$
\bqny{
v(t-s)=\int_0^{t-s}v^\prime(x)\td x\leq \int_{s}^{t}v^\prime(x)\td x = v(t)-v(s)
}
hence
\begin{equation*}
\frac{v(s)+v(t)-v(\abs{s-t})}{2} \geq \frac{v(s)+v(t)-\abs{v(s)-v(t)}}{2}.
\end{equation*}
Thus, inequality (\ref{g3}) holds which jointly with (\ref{g1}) and (\ref{g2})
implies
\begin{eqnarray*}
\pk{\exists_{t\in [0,T]}: \vk X(t)- \vk c t > \vk u}
&\leq& \pk{\exists_{t\in [0,T]}:  A \vk B(v(t)) - \vk c t > \vk u} \\
&=& \pk{\exists_{t\in [0,v(T)]}:  A \vk B(t) - \vk c w(t)  > \vk u}\\
&\le& \pk{\exists_{t\in [0,v(T)]}:  A \vk B(t) - \vk c \frac{T}{v(T)}t  > \vk u},
\end{eqnarray*}
where $w(t)$ is the inverse function of $v(t)$.
In the second line we used the fact that $v(\cdot)$ is continuous and strictly increasing, while the inequality in
the third line follows by concavity of $w(t)$ (recall that $v(t)$ is supposed to be convex).
Finally, by \cite[Theorem~1.1]{DHK21} (see also \cite{KoW20} for the centred case)
the above is bounded by
\[
\frac{\pk{\vk A \vk B(v(T))- \vk c T > \vk u}}
{\pk{\vk A \vk B(v(T))> \vk 0}}
=
\frac{\pk{ \vk X(T)- \vk c T > \vk u}}
{\pk{ \vk X(T)> \vk 0}}
.
\]
This completes the proof.
\end{proof}

\begin{proof}[Proof of Theorem \ref{th.exact}]
For any $L>0$ we have
\begin{align*}
P_3(u,L) \leq \frac{\pk{\exists_{t\in [0,T]}: \vk X(t)- \vk c t > \vk u}}{\pk{\vk X(T)- \vk c T >  u\vk a}}
\leq \sum_{n=0}^3 P_n(u,L),
\end{align*}
with
\begin{align*}
P_n(u,L) := \frac{\pk{\exists_{t\in [T_n(u),T_{n+1}(u)]}: \vk X(t)- \vk c t > \vk u}}{\pk{\vk X(T)- \vk c T >  u\vk a}},
\end{align*}
where $T_0(u) = 0$, $T_1(u)=T_1>0$ is chosen small enough to satisfy the conditions in Lemma \ref{lem:small},
\[T_2(u,L) := T-Lu^{-2}\ln^2 u, \quad T_3(u,L) := T-Lu^{-2}, \quad \text{and} \quad T_4(u) = T.\]

The proof consists of several steps which follow by lemmas displayed and proved in the rest of this section. It turns out, that asymptotically as $u\to\infty$, only $P_3(u,L)$ contributes to the asymptotics, while
$\sum_{n=0}^2 P_n(u,L)$ is asymptotically negligible. Since each term in  $\sum_{n=0}^2 P_n(u,L)$ needs a different argument for its negligibility, we provide detailed justification in separate lemmas. Namely,
\begin{itemize}
\item[$\diamond$] For any $L>0$ it holds that $\lim_{u\to\infty}P_0(u,L)=0$ due to Lemma \ref{lem:small},
\item[$\diamond$] For any $L>0$ it holds that $\lim_{u\to\infty}P_1(u,L)=0$ due to Lemma \ref{lem:mid},
\item[$\diamond$] $\lim_{L\to\infty}\lim_{u\to\infty}P_4(u,L)=0$ due to Lemma \ref{lem:gap},
\item[$\diamond$]  $\lim_{L\to\infty}\lim_{u\to\infty}P_5(u,L)=\mathcal{C}$ due to Lemma \ref{lem:A}.
\end{itemize}
This completes the proof.
\end{proof}

\begin{lem}\label{lem:small} Under the assumptions of Theorem \ref{th.exact} it holds that
\bqny{
\frac{\pk{\exists_{t\in[0,T_1]}:\vk X(t)-\vk c t>u\vk a}}{\pk{\vk X(T)-\vk c T>u\vk a}}\to 0, \quad u\to\infty
}
for all $T_1\in(0,T)$ small enough.
\end{lem}

\begin{lem}\label{lem:mid} Under the assumptions of Theorem \ref{th.exact}, for any $L>0$ it holds that
\bqny{
\frac{\pk{\exists_{t\in[T_1,T-Lu^{-2}\ln^2 u]}:\vk X(t)-\vk c t>u\vk a}}{\pk{\vk X(T)-\vk c T>u\vk a}}\to 0, \quad u\to\infty
}
for all $T_1\in(0,T)>0$ small enough.
\end{lem}

\begin{lem}\label{lem:gap}
Under the assumptions of Theorem \ref{th.exact}, there exist positive constants $C,\beta > 0$, such that
\bqny{
\lim_{u\to\infty} \frac{\pk{\exists_{t\in[T-Lu^{-2}\ln^2u,T-Lu^{-2}]}\vk X(t)-\vk ct>\vk au}}{\pk{\vk X(T)-\vk cT>u\vk a}}\leq Ce^{-\beta L}
}
for all $L>0$ large enough.
\end{lem}

\begin{lem}\label{lem:A}
Under the assumptions of Theorem \ref{th.exact}, for any $L > 0$ there exists a positive constant $\mathcal{C}(L)$ such that
\bqny{
\frac{\pk{\exists_{t\in[T-Lu^{-2},T]}:\vk X(t)-\vk c t >u\vk a}}{\pk{\vk X(T)-\vk c T>u\vk a}} \to  \mathcal{C}(L), \quad u\to\infty.
}
Moreover, $\lim_{L\to\infty}\mathcal C(L)=\mathcal{C}$, with $\mathcal{C}$ defined in \eqref{def:C}.
\end{lem}

The proofs of the four lemmas above are given in the following three subsections. The proof of each result is located at the very end of these subsections and is preceded by additional preparatory results.

\subsection{Proof of Lemma \ref{lem:small}}
Let $\varphi_{\vk X}(\cdot)$ be the pdf of random variable $\vk X$.
Before giving the proof of Lemma \ref{lem:small}, we need
the following lemma, which can be deduced from \cite[Lemma~4.2]{Has19}. Since we need a bit different
(although equivalent) form of
the derived below asymptotics, we provide an independent short proof of the following lemma.

\begin{lem}\label{prop:proj1} Let $\vk X\in\R^d$ be a centered Gaussian vector with an arbitrary, non-singular covariance matrix $\Sigma$. Then, for any $\vk c\in\R^d$, $\vk a\in \R^d \setminus (-\IF, 0]^d$ we have, as $u\to\infty$
\bqny{
\pk{\vk X - \vk c >u\vk a} \sim \frac{u^{-|I|}\varphi_{\vk X}(u\tilde{\vk a}+\vk c)}{\prod_{i\in I}\lambda_i}\int_{\R^{|J|}} \mathbb{I}\left\{\vk x_{U}<\vk 0_U\right\}e^{-\frac{1}{2}
\vk x_J^\top(\Sigma^{-1})_{JJ}\vk x_J}e^{\sprod{\tilde{\vk c}_J, \vk x_J}}\td\vk x_J,
}
where $\tilde{\vk c} := \vk c^\top\Sigma^{-1}$, and $\tilde{\vk a}, \vk\lambda$, and index sets $I,J,U$ corresponding to the quadratic programming problem $\Pi_\Sigma(\vk a)$ defined in \eqref{eq:QP}.
\end{lem}
\begin{proof}
Let $\bar u\in\R^d$ be such that $\bar u_i = u$ when $i\in I$ and $\bar u_i = 1$ when $i\in J$. We apply substitution $\vk w = u\tilde{\vk a} + \vk c - \vk x/\bar u$ and obtain
\begin{align*}
\pk{\vk X -\vk c >u\vk a} & = \int_{\vk w>u\vk a+\vk cT}\varphi_{\vk X}(\vk w)\td\vk w=u^{-\abs{I}}\int_{\vk x<u\bar{u}(\tilde{\vk a}-\vk a)}\varphi(u\tilde{\vk a}+\vk c-\vk x/\bar{\vk u})\td\vk x\\
& = u^{-|I|}\varphi(u\tilde{\vk a}+\vk c)\int_{\R^d}\mathbb{I}\left\{\vk x<u\bar{u}(\tilde{\vk a}-\vk a)\right\}\theta_u(\vk x)\td\vk x,
\end{align*}
where $\theta_u(\vk x) := \varphi(u\tilde{\vk a}+\vk c-\vk x/\bar{\vk u})/\varphi(u\tilde{\vk a}+\vk c)$. We have $\mathbb{I}_{\left\{\vk x<u\bar{u}(\tilde{\vk a}-\vk a)\right\}}\to\mathbb{I}\left\{\vk x_{I\cup U}<\vk 0\right\}$, as $u\to\infty$ and
\begin{align*}
\theta_u(\vk x) & = \exp\left\{u\tilde{\vk a}^\top\Sigma^{-1}(\vk x/\bar{\vk u}) + \vk c^\top\Sigma^{-1}(\vk x/\bar{\vk u}) - \frac{1}{2}(\vk x/\bar{\vk u})^\top \Sigma^{-1}(\vk x/\bar{\vk u})\right\} \\
& \to e^{\sprod{\vk \lambda_I, \vk x_I}} \cdot e^{-\frac{1}{2}\vk x_J^\top(\Sigma^{-1})_{JJ}\vk x_J}
e^{\sprod{\tilde{\vk c}_J, \vk x_J}} =: \theta(\vk x),
\end{align*}
as $u\to\infty$. So, applying the dominated convergence theorem, with dominating, integrable function
\[e^{\sprod{\vk \lambda_I, \vk x_I}}e^{\frac{1}{2}\sprod{\vk \lambda_I, |\vk x_I|}}
\cdot e^{-\frac{1}{2}\vk x_J^\top(\Sigma^{-1})_{JJ}\vk x_J}e^{\sprod{\tilde{\vk c}_J, \vk x_J}},\]
we obtain
\begin{align*}
\frac{\pk{\vk X-\vk c>u\vk a}}{u^{-|I|}\varphi(u\tilde{\vk a}+\vk c)}& \to \int_{\R^d}\mathbb{I}\left\{\vk x_{I\cup U}<\vk 0\right\}\theta(\vk x)\td\vk x\\
& = \int_{\R^{\abs{I}}}\mathbb{I}{\left\{\vk x_{I}<\vk 0_{I}\right\}}
e^{\sum\limits_{i\in I}\lambda_i x_i}\td\vk x_I \cdot \int_{\R^{|J|}}
\mathbb{I}\left\{\vk x_{U}<\vk 0_U\right\}e^{-\frac{1}{2}\vk x_J^\top(\Sigma^{-1})_{JJ}\vk x_J}e^{\sprod{\tilde{\vk c}_J, \vk x_J}}\td\vk x_J\\
& = \frac{1}{\prod_{i\in I}\lambda_i}\int_{\R^{|J|}} e^{-\frac{1}{2}\vk x_J^\top(\Sigma^{-1})_{JJ}\vk x_J}e^{\sprod{\tilde{\vk c}_J, \vk x_J}}\td\vk x_J,
\end{align*}
which concludes the proof.
\end{proof}

\begin{proof}[Proof of Lemma~\ref{lem:small}]
First, using \nelem{prop:proj1}, we know that there exist some $C>0$, $k\in\N$ such that
\begin{align*}
\pk{\vk X(T) - \vk c T > \vk au} \sim Cu^{-k}\varphi_T(\tilde{\vk a}u + \vk c T),
\end{align*}
as $u\to\infty$, where $\varphi_T$ is the density of $\vk X(T)$. Second, fix some $i\in\{1\ldot d\}$ such that $a_i>0$. Then
\bqny{
\pk{\exists_{t\in[0,T_1]}:\vk X(t)-\vk c t>u\vk a}
\leq
\pk{\exists_{t\in[0,T_1]}: X_i(t)- c_i t> a_i u}
\leq
\pk{\exists_{t\in[0,T_1]}: X_i(t)> a_i u-\abs{c_i}T_1}.
}
Using assumption {\bf B0}, we can apply Piterbarg's inequality \cite[Thm 8.1]{Pit96},
receiving for some positive constant $C_1$
and all sufficiently large $u$
\bqny{
\pk{\exists_{t\in[0,T_1]}: X_i(t) > a_i u-\abs{c_i}T_1}\leq \pk{\exists_{t\in[0,T_1]}: |X_i(t)| > a_i u-\abs{c_i}T_1}
\leq C_1 (a_i u-\abs{c_i}T_1)^{2}\pk{X_i(T_1) > a_i u-\abs{c_i}T_1}
}
for $u>0$. Hence, for all $u$ large enough we have
\bqny{
\frac{\pk{\exists_{t\in[0,T_1]}: X_i(t)> a_i u-\abs{c_i}T_1}}{\pk{\vk X(T) - \vk c T > \vk au}}
\leq 2 C_1(a_i u-\abs{c_i}T_1)^2u^{k}\frac{\pk{X_i(T_1) > a_i u-\abs{c_i}T_1}}{\varphi_T(\tilde{\vk a}u+\vk c T)}.
}
Since $\pk{\mathcal N(0,1) > u} \sim \frac{1}{\sqrt{2\pi}u}\exp(-u^2/2)$ as $u\to\infty$, it is left to show that
\begin{equation*}
\lim_{u\to\infty}u^{k -1}\frac{\exp\left(-(a_i u-\abs{c_i} T_1)^2/(2v_i(T_1))\right)}
{\exp\left(-\frac{1}{2}(\tilde{\vk a}u+\vk cT)^\top\Sigma^{-1}(T)(\tilde{\vk a}u+\vk cT)\right)}= 0.
\end{equation*}
We have
\begin{align*}
\frac{\exp\left(-(a_i u-\abs{c_i}T_1)^2/(2v_i(T_1))\right)}
{\exp\left(-\frac{1}{2}(\tilde{\vk a}u+\vk cT)^\top\Sigma^{-1}(T)(\tilde{\vk a}u+\vk cT)\right)} =
\exp\left[-\frac{1}{2}\left(\frac{a_i^2}{v_i(T_1)} - \tilde{\vk a}^\top \Sigma^{-1}(T)\tilde{\vk a}\right)u^2 + O(u)\right].
\end{align*}
Finally, since $v_i(0)=0$ and $v_i(\cdot)$ is continuous,
then $\frac{a_i^2}{v_i(T_1)}>\tilde{\vk a}^\top\Sigma^{-1}(T)\tilde{\vk a}$ for all $T_1$ small enough,
which completes the proof.
\end{proof}

\subsection{Proof of Lemma \ref{lem:mid}}

Before giving the proof, we need to layout preliminary results. Below, we cite the result from \cite[Lemma~4]{debicki2020extremes}. In the following, $J= \{1 \ldot d\}\setminus I$ can be empty; the claim in Lemma~\ref{lem:quadratic_programming}(ii) is formulated under the assumption that $J$ is non-empty.

\begin{lem}\label{lem:quadratic_programming}  Let $d \geq 2$ and
$\Sigma$ a $d \times d$ symmetric positive definite  matrix with inverse $\Sigma^{-1}$. If $\vk{a}\in\R^d \setminus (-\infty, 0]^d $, then the quadratic programming problem $\Pi_{\Sigma}(\vk b)$ defined in \eqref{eq:QP} has a unique solution $\tilde{\vk a}$ and there exists a unique non-empty index set $I\subset \{1 \ldot d\}$ with $|I| \le d$ elements such  that
\begin{itemize}
\item[(i)] $\tilde{\vk{a}}_{I} = \vk{a}_{I} \not=\vk{0}_I$;
\item[(ii)] $\tilde{\vk{a}}_{J} = \Sigma_{IJ}^{-1} \Sigma_{II}^{-1} \vk{a}_{I}\ge \vk{a}_{J}$, and $\Sigma_{II}^{-1} \vk{a}_{I}>\vk{0}_I$;
\item[(iii)] $\min_{\vk{x} \ge \vk{a}}\x^\top \Sigma^{-1}\x = \tilde{\vk{a}} ^\top \Sigma^{-1} \tilde{\vk{a}} = \vk{a} ^\top \Sigma^{-1} \tilde{\vk{a}} = \vk{a}_{I}^\top \Sigma_{II}^{-1}\vk{a}_{I}>0$,
\end{itemize}
with $\vk{\lambda}= \Sigma^{-1} \tilde{\vk{a}}$ satisfying $\vk{\lambda}_I= \Sigma_{II}^{-1} \vk{a}_I> \vk 0_I$, and $\vk{\lambda}_J= \vk{0}_J$.
\end{lem}

\begin{remark}\label{rem:change_a_with_a_tilde}
Using Lemma \ref{lem:quadratic_programming} it can be found that
$\tilde{\vk{a}} ^\top \Sigma^{-1} \tilde{\vk{a}} = \vk{a} ^\top \Sigma^{-1} \tilde{\vk{a}}$.
\end{remark}

To the end of this paper, let $D(t):=:= \tilde{\vk a}(t)^\top\Sigma^{-1}(t)\tilde{\vk a}(t)$.
\begin{lem}\label{prop:D} Assuming conditions \textbf{B0-BII} hold, then $D(t)$ is positive and strictly decreasing on $t\in(0,T]$. Moreover, $\dot D(T) = -\left\|{\rm diag}\left(\sqrt{\dot{\vk v}(T)}/\vk v(T)\right)A^{-1}\tilde{\vk a}(T)\right\|_2^2<0$.
\end{lem}

\begin{proof}
Let $0< t_1<t_2\leq T$. Then
\begin{align*}
D(t_2) &= \tilde{\vk a}(t_2)^\top\Sigma^{-1}(t_2)\tilde{\vk a}(t_2)
\leq
\tilde{\vk a}(t_1)^\top\Sigma^{-1}(t_2)\tilde{\vk a}(t_1)=\tilde{\vk a}(t_1)^\top {A^{-\top}}~{\rm diag}(1/\vk v(t_2))A^{-1}\tilde{\vk a}(t_1)\\
&=D(t_1)-\tilde{\vk a}(t_1)^\top {A^{-1}}^\top~{\rm diag}\left(\frac{\vk v(t_2)-\vk v(t_1)}{\vk v(t_2)\vk v(t_1)}\right)A^{-1}\tilde{\vk a}(t_1)\\
&=D(t_1)-\left\| {\rm diag}\left(\frac{\sqrt{\vk v(t_2)-\vk v(t_1)}}{\sqrt{\vk v(t_2)\vk v(t_1)}}\right)A^{-1}\tilde{\vk a}(t_1)\right\|_2^2<D(t_1)
\end{align*}
because $\tilde{\vk a}(t_1)\not=\vk 0$ and $\vk v(t)$ is strictly increasing. This shows that $D(t)$ is strictly decreasing. Furthermore, we have
\bqny{
\tilde{\vk a}(T)^\top {A^{-\top}}{\rm diag}\left(\frac{\vk v(t)-\vk v(T)}{(T-t)\vk v(t)\vk v(T)}\right)A^{-1}\tilde{\vk a}(T)\leq\frac{D(T)-D(t)}{T-t}\leq\tilde{\vk a}(t)^\top {A^{-\top}}{\rm diag}\left(\frac{\vk v(t)-\vk v(T)}{(T-t)\vk v(t)\vk v(T)}\right)A^{-1}\tilde{\vk a}(t)
}
using that $\tilde{\vk a}(t)$ and $\vk v(t)$ are continuous and $\vk v(t)$ has a positive derivative at the point $t=T$, we have as $t\to T$
\bqny{
\frac{D(T)-D(t)}{T-t}\to-\tilde{\vk a}(T)^\top {A^{-\top}}\left(\frac{\dot{\vk v}(T)}{\vk v^2(T)}\right)A^{-1}\tilde{\vk a}(T)=-\left\|{\rm diag}\left(\sqrt{\dot{\vk v}(T)}/\vk v(T)\right)A^{-1}\tilde{\vk a}(T)\right\|_2^2<0,
}
hence the claim follows.
\end{proof}

\begin{lem}\label{lem:lipshtz_continuity}
Let $T_0\in(0,T]$. Then $D(t)$, and $\tilde{a}_i(t)$, $\lambda_i(t)$, $\frac{\lambda_i(t)}{D(t)}$ are Lipshitz continuous functions on $t\in[T_0,T]$ for all $i\in\{1,\ldots,d\}$.
\end{lem}

\begin{proof}
Let $T_0\in(0,T]$ be fixed. According to \cite[Theorem~3.1]{Hag79}, $\tilde a_i(\cdot)$ and $\lambda_i(\cdot)$ are Lipshitz continuous, provided that conditions \cite[A1-A3]{Hag79} are satisfied. First, let us note that the conditions A1-A2 are clearly satisfied in our setting so will will focus only on condition A3. In order to state what is condition A3, let $M(t)\in\R^{I(t)\times d}$ such that $M(t) := (-\mathcal I_d)_{I(t)}$, where $(-\mathcal I_d)_{I(t)}$ is the submatrix of $-\mathcal I_d$ consisting of rows corresponding to the indices of $I(t)$. Then, condition  \cite[A3]{Hag79} states that that there exist $\alpha,\beta>0$ such that for all $t\in[T_0,T]$:
\begin{itemize}
\item[(i)] $x^\top \Sigma^{-1}(t) x \geq \alpha \|x\|^2$ for all $x\in\R^d$ satisfying $M(t)x = 0$,
\item[(ii)] $\|M(t)^\top x\| \geq \beta\|x\|$ for all $x\in\R^{I(t)}$
\end{itemize}
Since $\|M(t)^\top x\| = \|x\|$, then (ii) is satisfied with $\beta=1$. To see that (i) holds, we have
\bqny{
x^\top \Sigma^{-1}(t)x \geq \sigma_{1}(t)\|x\|^2
}
where $\sigma_{1}(t)$ is the smallest  eigenvalue of $\Sigma^{-1}(t)$. The matrix $\Sigma^{-1}(t)$ is symmetric and positive definite for $t>0$, thus it has only real positive eigenvalues $\sigma_1(t)\ldots \sigma_d(t)$. So the related characteristic polynomial $p_{\Sigma^{-1}(t)}$ has continuous monoms and always has $d$ real solutions. It means that we can order the eigenvalues $\sigma_1(t)\ldot \sigma_d(t)$ in such way that this functions will be continuous by $t$ and thus we can take $\alpha := \min_{t\in[T_0,T]} \sigma_1(t)>0$, which concludes the proof of (i) and of the Lipshitz continuity of $\tilde a_i(\cdot)$ and $\lambda_i(\cdot)$.

Now, the fact that $\tilde a_i(\cdot)$ is Lipshitz continuous immediately implies the Lipshitz continuity of $D(\cdot)$. Lastly, we need to show the Lipshitz continuity of $\lambda_i(t)/D(t)$. For $t,s\in[T_0,T]$ we have
\begin{align*}
\abs{\frac{\lambda_i(s)}{D(s)}-\frac{\lambda_i(t)}{D(t)}} \leq \abs{\frac{\lambda_i(s)-\lambda_i(t)}{D(s)}} + \lambda_i(t) \abs{\frac{D(t)-D(s)}{D(t)D(s)}}.
\end{align*}
The proof is concluded by the Lipshitz continuity of $\lambda_i(\cdot)$, $D(\cdot)$, and by noting that $\min_{t\in[T_0,T]} D(t) = D(T) > 0$; see Lemma \ref{prop:D}.
\end{proof}

In the following, $\argmin_{t\in [a,b]} f(t)$ is the smallest minimizer of function $f(\cdot)$ over set $[a,b]$.
For $t\in(0,T]$ we define function
\begin{equation}\label{def:G}
\mathcal G(t) := \frac{\sprod{\vk \lambda(t),\vk c}t}{D(t)}.
\end{equation}

\begin{lem}\label{lem:bound}
Let $\vk X(t)=A\vk Z(t),~t\in[0,T]$ be such that $\vk Z(t)$ satisfies {\bf B0}-{\bf BII} and $\vk a\in\R^d\setminus(-\infty,0]^d$. Then, for any $T_1\in(0,T)$ there exists a constant $C>0$ such that for any $T_1 \leq L < R \leq T$ and $u > -\mathcal G(T^*)$
\bqny{
\pk{\exists_{t\in[L,R]}:\vk X(t)-\vk c t>u\vk a}\leq C(u+\mathcal G(T^*))\exp\bigl(-D(R)(u+\mathcal G(T^*))^2/2\bigr)
}
where $T^* := \argmin_{t\in[L,R]} \mathcal G(t)$.
\end{lem}

\begin{proof}
Recall that $\Sigma(t) := \var(\vk X(t))$ and $\tilde{\vk a}(t)$ is the solution to the quadratic programming problem $\Pi_{\Sigma(t)}(\vk a)$ for each $t>0$, as in \eqref{eq:QP}. According to Lemma~\ref{lem:quadratic_programming}(iii), for $t>0$ we have $D(t) > 0$. Hence
\begin{align}
\nonumber\pk{\exists_{t\in[L,R]}:\vk X(t)-\vk ct>u\vk a}&\leq\pk{\exists_{t\in[L,R]}:\sprod{\vk\lambda(t),(\vk X(t)-\vk c t)}>\sprod{\vk\lambda(t),\vk a} u}\\
\label{eq:lembound_1}&=\pk{\exists_{t\in[L,R]}:\frac{\sprod{\vk\lambda(t),\vk X(t)}}{D(t)}>u+\mathcal G(t)},
\end{align}
with $\mathcal G(\cdot)$ defined in \eqref{def:G}. In the following let $Y(t) := \frac{\sprod{\vk\lambda(t),\vk X(t)}}{D(t)}$. Using the inequality $(\sum_{i=1}^d a_i)^2 \leq d\sum_{i=1}^d a_i^2$ we find that
\begin{align*}
\E{\left(Y(t)-Y(s)\right)^2}&\leq 2\E{\left(\frac{\sprod{\vk\lambda(t),\vk X(t)}}{D(t)}-\frac{\sprod{\vk\lambda(s),\vk X(t)}}{D(s)}\right)^2}+2\E{\left(\frac{\sprod{\vk\lambda(s),\vk X(t)}}{D(s)}-\frac{\sprod{\vk\lambda(s),\vk X(s)}}{D(s)}\right)^2}\\
&\leq 2d\sum_{i=1}^{d}\left(\frac{\lambda_i(t)}{D(t)}-\frac{\lambda_i(s)}{D(s)}\right)^2v_i(t)+2d\frac{\lambda_i^2(s)}{D^2(s)}v_i(|t-s|).
\end{align*}
Now, the functions $v_i(\cdot)$ and $\lambda_i(\cdot)/D(\cdot)$ are Lipshitz continuous due to {\bf B0} and Lemma \ref{lem:lipshtz_continuity} respectively, so there exists $C_1>0$ such that $|v_i(|t-s|)| \leq C_1|t-s|$, and $|\frac{\lambda_i(t)}{D(t)}-\frac{\lambda_i(s)}{D(s)}| \leq C_1|t-s|$ for all $i\in\{1,\ldots,d\}$, thus
\begin{align*}
\E{\left(Y(t)-Y(s)\right)^2} \leq 2d^2C_1^2|t-s|^2 + 2d\max_{i\in\{1,\ldots,d\}}\max_{t\in[L,R]}\left\{\frac{\lambda_i^2(t)}{D^2(t)}\right\} C_1|t-s|.
\end{align*}
We conclude that there exists $C_2 > 0$ such that
\begin{align*}
\E{\left(Y(t)-Y(s)\right)^2} \leq C_2|t-s|
\end{align*}
for all $t,s\in [L,R]$. Since $\var(Y(t)) = \var(\tilde{\vk a}^\top(t)\Sigma^{-1}\vk X(t))/D^2(t) = 1/D(t)$ and $D(t)$
is strictly decreasing, see Lemma~\ref{prop:D}, then the maximum of $\var(Y(t))$ over $[L,R]$ is attained at $t=R$.
According to Piterbarg inequality \cite[Thm 8.1]{Pit96}, there exists a constant $C_3$ such that for any $0<L<R\leq T$ we have
\begin{equation}\label{eq:lembound_2}
\pk{\exists_{t\in[L,R]}:Y(t)>u}\leq C_3(R-L) u^2 \pk{Y(R) > u}
\end{equation}
for $u>0$. Finally, since $\pk{\mathcal{N}(0,1)>u} \leq \frac{1}{\sqrt{2\pi}u}e^{-u^2/2}$ for $u>0$, then upon combining \eqref{eq:lembound_1}, \eqref{eq:lembound_2} we obtain
\begin{align*}
\pk{\exists_{t\in[L,R]}:\vk X(t)-\vk ct>u\vk a} \leq C_3(R-L)\sqrt{\frac{D(R)}{2\pi}}(u+\mathcal G(T^*))\exp\{-D(R)(u+\mathcal G(T^*))^2/2\}.
\end{align*}
Since function $D(\cdot)$ is decreasing (Lemma~\ref{prop:D}),
we conclude the proof by taking $C := C_3(R-L)\sqrt{\frac{D(T_1)}{2\pi}}$.
\end{proof}

\begin{lem}\label{lem:exp.bound}
Let $\textbf{B0-BII}$ hold and let $T_1\in(0,T]$ and $T_1 \leq L(u) < R(u) \leq T$. If either of the following two conditions is satisfied:
\begin{itemize}
\item[(i)] $R(u)\to T$, and $u(T-R(u))\to\infty$, or
\item[(ii)] $T-L(u)=o(u(T-R(u)))$, and $u^2(T-R(u))/\ln(u) \to \infty$,
\end{itemize}
then
\begin{align*}
\lim_{u\to0}\frac{\pk{\exists_{t\in[L(u),R(u)]}:\vk X(t)-\vk ct>u\vk a}}{\pk{\vk X(T) - \vk c T > u\vk a}} = 0.
\end{align*}
as $u\to\infty$.
\end{lem}

\begin{proof}
Using \nelem{prop:proj1}, we know that there exist some $C_1>0$, $k\in\N$ such that
\begin{align}\label{eq:lem:exp.bound1}
\pk{\vk X(T) - \vk c T > \vk au} \sim C_1u^{-k}\varphi_T(\tilde{\vk a}u + \vk c T),
\end{align}
as $u\to\infty$, where $\varphi_T$ is the density of $\vk X(T)$.
According to Lemma~\ref{lem:bound} there exist $C_2 > 0, K\in\R$ such that,
for all $u > -\mathcal G(T^*(u))$ we have
\begin{equation}\label{eq:lem:exp.bound2}
\pk{\exists_{t\in[L(u),R(u)]}:\vk X(t)-\vk c t>u\vk a}\leq C_2(u+\mathcal G(T^*(u))\exp\bigl(-D(R(u))(u+\mathcal G(T^*(u)))^2/2\bigr),
\end{equation}
where $K(\cdot)$ is defined in Lemma \ref{lem:bound} and $T^*(u) := \argmin_{t\in[L(u),R(u)]} \mathcal G(T)$. From now on, we take $u > -\inf_{t\in[T_1,T]} \mathcal G(t)$, which is finite due to Lemma \ref{lem:lipshtz_continuity}.

For brevity, in the following we denote $\Sigma := \Sigma(T)$, $\vk{\tilde a} := \vk{\tilde a}(T)$. In light of \eqref{eq:lem:exp.bound1} and \eqref{eq:lem:exp.bound2}, it suffices to show that, for any $\beta \in\R$ we have
\begin{equation}\label{eq:to_prove:lem_exp}
u^\beta\frac{\exp(-D(R(u))(u+\mathcal G(T^*(u)))^2/2)}{\exp\left(-(\tilde{\vk a}u+\vk cT)^\top\Sigma^{-1}(\tilde{\vk a}u+\vk cT)/2\right)} = \exp\Big(-\big(u^2h_2(u) + uh_1(u) + h_0(u)\big) + \beta\ln(u)\Big) \to 0,
\end{equation}
as $u\to\infty$, where
\begin{align*}
h_2(u) = \frac{1}{2}\Big(D(R(u))-D(T)\Big), \quad h_1(u) = D(R(u))\mathcal G(T^*(u))-D(T)\mathcal G(T), \quad h_0(u) := \frac{1}{2}\Big(D(R(u))K^2(T^*(u))-T^2\vk c^\top\Sigma^{-1}\vk c\Big).
\end{align*}
We notice that functions $|h_i(u)|$ are all bounded for $u$ large enough.

Suppose that $L(u), R(u)$ satisfy conditions (i). Due to the continuity of $D(\cdot)$, we have that $D(R(u))\to D(T)$. Using the assumption {\bf B1}, that $D(t)$ is differentiable at the point $t=T$, with $\dot D(T) < 0$, we have
\begin{align*}
-uh_2(u) - h_1(u) = u(T-R(u)) \cdot \frac{D(R(u)) - D(T)}{2(R(u)-T)} - h_1(u) \to -\infty,
\end{align*}
which implies \eqref{eq:to_prove:lem_exp} under conditions in item (i).

Next, suppose that $L(u), R(u)$ satisfy  conditions (ii). We have
\bqny{
h_1(u) = \tfrac{D(R(u))}{D(T^*(u))} \cdot \Big(D(T^*(u))\mathcal G(T^*(u)) - D(T)\mathcal G(T)\Big) + \tfrac{D(T)}{D(T^*(u))}\mathcal G(T) \Big(D(R(u)) - D(T^*(u))\Big).
}
Functions $D(t)$ and $D(t)\mathcal G(t) = \sprod{\vk{\lambda}(t),c}t$ are Lipshitz continuous due to Lemma \ref{lem:lipshtz_continuity} and the fact that sums and products of Lipshitz continuous functions are Lipshitz continuous. This implies that there exists $C_3>0$ such that
\begin{align*}
\Big|D(T^*(u))\mathcal G(T^*(u)) - D(T)\mathcal G(T)\Big| \leq C_3|T-T^*(u)| \leq C_3|T-L(u)|,
\end{align*}
as well as
\begin{align*}
\Big|D(R(u)) - D(T^*(u))\Big| \leq C_3|R(u)-T^*(u)| \leq C_3|T-L(u)|.
\end{align*}
Hence, there exists $C_4>0$ such that $|h_1(u)| \leq C_4|T-L(u)|$ for all $u$ large enough and
\begin{align*}
-u^2h_2(u) - uh_1(u) \leq u^2(T-R(u)) \cdot \left[\frac{D(R(u)) - D(T)}{2(R(u)-T)} + C_4 \cdot \frac{T-L(u)}{u(T-R(u))}\right]
\end{align*}
Since $D(R(u))<D(T)$ and $\dot D(T) < 0$ and $T-L(u)=o(u(T-R(u)))$ then the term in the square brackets above is eventually negative and bounded away from $0$ for $u$ large enough. Finally, \eqref{eq:to_prove:lem_exp} follows from $u^2(T-R(u))/\ln(u) \to \infty$.
\end{proof}

\begin{proof}[Proof of Lemma \ref{lem:mid}]
Consider the following upper bound
\begin{eqnarray*}
\lefteqn{\frac{\pk{\exists_{t\in[T_1,T-Lu^{-2}\ln^2 u]}:\vk X(t)-\vk c t>u\vk a}}{\pk{\vk X(T)-\vk c T>u\vk a}}\leq \frac{\pk{\exists_{t\in[T_1,T-Lu^{-1}\ln u]}:\vk X(t)-\vk c t>u\vk a}}{\pk{\vk X(T)-\vk c T>u\vk a}}}\\
&& + \frac{\pk{\exists_{t\in[T-Lu^{-1}\ln u,T-Lu^{-1}\ln^{-1} u]}:\vk X(t)-\vk c t>u\vk a}}{\pk{\vk X(T)-\vk c T>u\vk a}} + \frac{\pk{\exists_{t\in[T-Lu^{-1}\ln^{-1} u],T-Lu^{-2}\ln^2 u]}:\vk X(t)-\vk c t>u\vk a}}{\pk{\vk X(T)-\vk c T>u\vk a}}.
\end{eqnarray*}
Now, the first and second term above satisfy condition (i) of Lemma~\ref{lem:exp.bound},
while the third term satisfies condition (ii) of Lemma~\ref{lem:exp.bound}.
Thus the righthand side of the above inequality converges to $0$, as $u\to\infty$.
\end{proof}

\subsection{Proofs of Lemma \ref{lem:gap} and Lemma \ref{lem:A}}
Before proceeding to the proof of Lemma \ref{lem:gap} and Lemma \ref{lem:A} we need some
auxiliary lemmas.
The following result generalizes \cite[Lemma~5.3]{debicki2020extremes}.

\begin{lem}\label{lem:int_repr}
For any $\vk f := (f_1,\ldots,f_d)\in R^d$
\begin{align*}
\int_{\R^d}\mathbb{I}\left\{\exists_{t\in[0,L]}:~\vk x<\vk f t\right\}e^{\sum_{i=1}^d x_i}\td{\vk x} = \begin{cases}
1+\sum_{i=1}^d f_i^+L,\qquad &\sum_{i=1}^d f_i = 0, \\
\frac{\sum_{i=1}^d f_i^-}{\sum_{i=1}^d f_i}+\frac{\sum_{i=1}^d f_i^+}{\sum_{i=1}^d f_i}e^{\sum_{i=1}^d f_iL},\qquad &\sum_{i=1}^d f_i\not = 0,
\end{cases}
\end{align*}
where $f_i^+ := \max\{f_i, 0\}$ and $f_i^- := \min\{f_i, 0\}$.
\end{lem}

\begin{proof}
Define $S_+ := \sum_{i=1}^d f_i^+$, $S_- := -\sum_{i=1}^d f_i^-$, and $S := \sum_{i=1}^d f_i = S_+ - S_-$. Without the loss of generality, let $k\in\{0,\ldots,d\}$ be such that $f_1\geq 0, \ldots, f_k \geq 0$, and $f_{k+1}<0,\ldots, f_{d}<0$.  We distinguish three cases: (i) $k=d$ (all $f_i$s are non-negative, which is equivalent to $S_- = 0$), (ii) $k=0$ (all $f_i$s are negative, which is equivalent to $S_+ = 0$), and (iii) $0 < k < d$. It can be easily seen that in case (i) we have
\begin{align*}
\int_{\R^d}\mathbb{I}\left\{\exists_{t\in[0,L]}:~\vk x<\vk f t\right\}e^{\sum_{i=1}^d x_i}\td\vk x = \prod_{i=1}^d\int_{-\infty}^{f_i L}e^{x_i}\td x_i = e^{SL},
\end{align*}
and in case (ii) we have
\begin{align*}
\int_{\R^d}\mathbb{I}\left\{\exists_{t\in[0,L]}:~\vk x<\vk f t\right\}e^{\sum_{i=1}^d x_i}\td\vk x = \prod_{i=1}^d\int_{0}^{\infty}e^{-x_i}\td x_i = 1.
\end{align*}
Till the end we consider case (iii). Let us define
\begin{align*}
Q_1 := \{\vk x\in\R^d : \forall_{i\in\{1,\ldots,d\}} \text{ if } f_i < 0 \text{ then } x_i < 0\}, \quad Q_2 := \{x\in\R^d :\exists_{i\in\{1,\ldots,d\}} \text{ if } f_i < 0 \text{ then } x_i \geq 0\},
\end{align*}
so that $Q_1 \cup Q_2 = \R^d$ and $Q_1 \cap Q_2 = \emptyset$. It can be seen that
\begin{align*}
\int_{Q_2}\mathbb{I}\left\{\exists_{t\in[0,L]}:~\vk x<\vk f t\right\}e^{\sum_{i=1}^d x_i}\td\vk x = 0.
\end{align*}
Furthermore, with $m := m(x_{k+1},\ldots, x_d) = \min\{\tfrac{x_{k+1}}{-f_{k+1}},\ldots, \tfrac{x_{d}}{-f_{d}}\}$, we have
\begin{align*}
\int_{Q_1}\mathbb{I}\left\{\exists_{t\in[0,L]}:~\vk x<\vk f t\right\}e^{\sum_{i=1}^d x_i}\td\vk x & = \int_0^\infty\cdots\int_0^\infty \left[\prod_{i=1}^k\int_{-\infty}^{f_i(L\wedge m)}e^{x_i}\td x_i\right] e^{-\sum_{i={k+1}}^d x_{i}}\,\td x_{k+1}\cdots\td x_d \\
& = \int_0^\infty\cdots\int_0^\infty \exp\left\{S_+(L \wedge m)\right\} e^{-\sum_{i={k+1}}^d x_{i}}\td x_{k+1}\cdots\td x_d.
\end{align*}
We recognize that $\exp\{-\sum_{i={k+1}}^d x_i\}\cdot\mathbb{I}\left\{x_i \geq 0\right\}$ is the density of
minimum of
$d-k$ independent exponential distributions with rate $1$;
using that such minimum is again exponentially distributed with rate $(d-k)$ 
we find that, with $Y\sim {\rm Exp}(S_-)$,
\begin{align*}
\int_{Q_1}\mathbb{I}\left\{\exists_{t\in[0,L]}:~\vk x<\vk f t\right\}e^{\sum_{i=1}^d x_i}\td x & = \E{e^{S_+(L \wedge Y)}} = S_- \int_0^\infty e^{S_+(L \wedge y)-S_- y}\td y \\
& = S_- \int_0^L e^{(S_+ - S_-)y}\td y + e^{S_+L}\int_L^\infty S_- e^{-S_-y}\td y \\
& = \begin{cases}
L\cdot S_- + e^{SL}, & S = 0 \\
\frac{S_-}{S} \cdot (e^{S\cdot L}-1) + e^{SL}, & \text{otherwise},
\end{cases}
\end{align*}
which completes the proof.
\end{proof}

\begin{lem}\label{lem:tau}
There exist $\bar\tau\in(0,T)$, $\lambda^*>0$, $\eta > 0$ such that:
\begin{itemize}
\item[(i)] $I_T \subseteq I_t$ for all $t\in[\bar\tau,T]$,
\item[(ii)] $\lambda_i(t)>\lambda^*$ for all $i\in I_t$, $t\in[\bar\tau,T]$, and
\item[(iii)] $\Sigma^{-1}(t)-\eta \mathcal{I}_d$ is positive definite for all $t\in[\bar\tau,T]$.
\end{itemize}
\end{lem}

\begin{proof}
According to Lemma~\ref{lem:quadratic_programming}, $i\in I_{t}$ if and only if $\lambda_i(t)>0$. Since $\lambda_i(t)$ is a continuous function for any $i\in\{1\ldot d\}$ (see  Lemma~\ref{lem:lipshtz_continuity}), then for any $i\in I_T$ there must exist $\tau_i<T$, and $\lambda_i^*>0$ such that $\lambda_i(t) > \lambda_i^*$ for all $t\in[\tau_i,T]$. for all $t\in[\tau_i,T]$ we have $\lambda_i(t)>0$. Thus the claims in (i) and (ii) follow by taking $\tau=\max_{i\in I_T}(\tau_i)$ and $\lambda^*=\min_{i\in I_T}\lambda^*_i$.

The matrix $\Sigma^{-1}(t)$ is symmetric and positive definite for $t>0$, thus it has only real positive eigenvalues $\sigma_1(t)\ldots \sigma_d(t)$. So the related characteristic polynomial $p_{\Sigma^{-1}(t)}$ has continuous monoms and always has $d$ real solutions. It means that we can order the eigenvalues $\sigma_1(t)\ldot \sigma_d(t)$ in such way that this functions will be continuous by $t$ and thus $\eta=\min\limits_{i\in \{1\ldot d\}}\min\limits_{t\in[\tau,T]}\sigma_i(t)>0$. This concludes the proof of (iii).
\end{proof}

In the following, for all $u>0, \tau\in(0,T]$, and $\vk x\in\R^d$ we define:
\begin{equation}\label{def:w}
\vk w_{u,\tau}(\vk x):=u\tilde{\vk a}(\tau)+\vk c\tau - \frac{\vk x}{\bar{\vk u}(\tau)}
\end{equation}
where $\bar{\vk u}(\tau)\in\{u,1\}^d$ such that $\bar{\vk u}_{I_\tau}(\tau) := u\cdot{\vk 1}_{|I_\tau|}$, and $\bar{\vk u}_{J_\tau}(\tau) := {\vk 1}_{|J_\tau|}$, that is $\bar{\vk u}(\tau)$ has the components in the set $I_\tau$ equal to $u$ and the other components equal to 1. Further, for all $L>0$, $\tau\in(0,T]$, $\vk x\in\R^d$ and $u > \sqrt{L/\tau}$ consider a Gaussian process $\{\vk Z^{\vk x}_{u, \tau}(t), t\in[0,L]\}$ defined conditionally:
\begin{equation}\label{eq:def_Z_u}
\Big(\vk Z^{\vk x}_{u, \tau}(t)\Big)_{t\in[0,L]} :\overset{d}{=} \Big(\vk Z(\tau-\tfrac{t}{u^2}) \mid \vk Z(\tau)=A^{-1}\vk w_{u,\tau}(\vk x) \Big)_{t\in[0,L]}
\end{equation}
Since the components of $\vk Z(t)$ are mutually independent, then the components of $\vk Z^{\vk x}_{u, \tau}(t)$ are mutually independent as well, i.e. $\cov((\vk Z^{\vk x}_{u, \tau}(s))_{i}, (\vk Z^{\vk x}_{u, \tau}(t))_j) = 0$ for $i\neq j$. By the definition in \eqref{eq:def_Z_u}, for any $i\in\{1,\ldots,d\}$, $t,s\in[0,L]$ we have
\begin{align*}
\E{(\vk Z^{\vk x}_{u, \tau}(t))_{i}} & = \frac{\rho_i(\tau-\tfrac{t}{u^2}, \tau)}{v_i(\tau)} \left(A^{-1}\vk w_{u,\tau}(\vk x)\right)_i\\
\cov\Big((\vk Z^{\vk x}_{u, \tau}(s))_{i}, (\vk Z^{\vk x}_{u, \tau}(t))_i)\Big) & = \rho_i(\tau-\tfrac{s}{u^2}, \tau-\tfrac{t} {u^2}) - \frac{\rho_i(\tau-\tfrac{s}{u^2}, \tau)\rho_i(\tau-\tfrac{t}{u^2}, \tau)}{v_i(\tau)}
\end{align*}
with $\rho_i$ defined in \eqref{eq:def_rho_i}. In the following let $\widehat{\vk Z}_{u, \tau}(t) := \vk Z^{\vk x}_{u, \tau}(t) - \E{\vk Z^{\vk x}_{u, \tau}(t))}$. It is noted that the distribution of $\widehat{\vk Z}_{u, \tau}(t)$ does not depend on $\vk x$.

\begin{lem}\label{lem:piterbarg_condition_Zhat}
There exists a constant $C>0$ such that for all $L>0$, $\tau\in(0,T]$, and $t,s\in[0,L]$ we have
\bqny{
u^2\E{\Big((A\hat{\vk Z}_{u, \tau}(t))_{i}- (A\hat{\vk Z}_{u, \tau}(s))_{i}\Big)^2}\leq Cu^2\max_{j\in\{1,\ldots,d\}} v_j(Lu^{-2})
}
for all $i\in\{1,\ldots,d\}$ and $u$ large enough.
\end{lem}
\begin{proof}
For brevity, in the following denote $\overline t := \tau-\tfrac{t}{u^2}$, $\overline s := \tau-\tfrac{s}{u^2}$, and $\hat{\vk Z}_u(t) := \hat{\vk Z}_{u,\tau}(t)$. We have
\begin{align*}
\E{\left((\hat{\vk Z}_u(t))_{i}- (\hat{\vk Z}_u(s))_{i}\right)^2} & = \var\{(\hat{\vk Z}_u(t))_{i}\} + \var\{(\hat{\vk Z}_u(s))_{i}\} - 2\cov\{(\hat{\vk Z}_u(t))_{i},(\hat{\vk Z}_u(s))_{i}\} \\
& = v_i(\bar t) - \frac{\rho^2_i(\bar t, \tau)}{v_i(\tau)} + v_i(\bar s) - \frac{\rho^2_i(\bar s, \tau)}{v_i(\tau)} - 2\left(\rho_i(\bar s,\bar t)- \frac{\rho_i(\bar t, \tau)\rho_i(\bar s, \tau)}{v_i(\tau)}\right) \\
& = v_i(|\bar s- \bar t|) - \frac{\big(\rho_i(\bar t, \tau) - \rho_i(\bar s, \tau)\big)^2}{v_i(\tau)}.
\end{align*}
Now, the above is not greater than $v_i(|\bar s - \bar t|) = v_i(|s-t|/u^2) \leq v_i(Lu^{-2})$. Furthermore, we have
\begin{align*}
u^2\E{\left((A\hat{\vk Z}_u(s))_i - (A\hat{\vk Z}_u(t))_i\right)^2} & =  u^2\E{\Big(\sum_{j=1}^d a_{ij}\big((\hat{\vk Z}_u(s))_j - (\hat{\vk Z}_u(t))_j\big)\Big)^2} \\
& \leq u^2\left(\sum_{j=1}^d a_{ij}^2\right) \left(\sum_{j=1}^d\E{(\hat{\vk Z}_u(s))_j - (\hat{\vk Z}_u(t))_j)^2}\right),
\end{align*}
where we used Cauchy–Schwarz inequality. This completes the proof.
\end{proof}
The following corollary to
Lemma \ref{lem:piterbarg_condition_Zhat}
is a straightforward application of Piterbarg inequality \cite[Theorem~8.1]{Pit96}
and Lipshitz continuity of functions $v_i(\cdot)$.
\begin{korr}\label{cor:weak_convergence_to_0}
There exists $C>0$ such that for all $L>0$, $\tau\in(0,T]$, $z>0$ we have
\begin{align*}
\pk{\sup_{t\in[0,L]} u(A\hat{\vk Z}_{u, \tau}(t))_{i} > z} \leq Cz^2e^{-z^2/(2u^2\max_{j\in\{1,\ldots,d\}} v_j(Lu^{-2}))}
\end{align*}
for $i\in\{1,\ldots,d\}$ and all $u$ large enough.
\end{korr}

In the following, for any $i\in\{1,\ldots,d\}$, $t,s\in[0,L]$ we define
\begin{equation}\label{def:h,theta}
\begin{split}
h_{u,\tau}(L, \vk x) &:= \pk{\exists_{s\in[\tau-Lu^{-2},\tau]}:\vk X(t)-\vk c t>u\vk a \mid \vk X(\tau)=\vk w_{u,\tau}(\vk x)},\\
\theta_{u,\tau}(\vk x) &:= \varphi_\tau(\vk w_{u,\tau}(\vk x))/\varphi_\tau(u\tilde{\vk a}(\tau) +\vk c\tau)
\end{split}
\end{equation}

\begin{lem}\label{lem:dominant}
There exists $\tau_0\in(0,T)$ and function $H:\R_+\times \R^d \to \R_+$ satisfying $\int_{\R^d} H(L,\vk x)\td\vk x =: C^*(L) <  \infty$ for all $L>0$, and $u_0 : \R_+\to\R_+$ such that for all $L>0$, $\tau\in[\tau_0,T]$, and $\vk x\in\R^d$ we have
\begin{align*}
h_{u,\tau}(L, \vk x)\theta_{u,\tau}(\vk x) \leq H(L,\vk x)
\end{align*}
for all $u>u_0(L)$. Moreover, there exists $C^*>0$ such that $\limsup_{L\to\infty} C^*(L) < C^*$.
\end{lem}

\begin{proof}
For $u>0, \tau\in(0,T]$, and $\vk x\in\R^d$ let $\theta_{u,\tau}(\vk x) := \varphi_\tau(\vk w_{u,\tau}(\vk x))/\varphi_\tau(u\tilde{\vk a}(\tau) +\vk c\tau)$. Then
\begin{align*}
\theta_{u,\tau}(\vk x) & = \exp\left\{u\tilde{\vk a}^\top(\tau)\Sigma^{-1}(\tau)(\vk x/\bar{\vk u}) + \tau\vk c^\top\Sigma^{-1}(\tau)(\vk x/\bar{\vk u}) - \frac{1}{2}(\vk x/\bar{\vk u})^\top \Sigma^{-1}(\tau)(\vk x/\bar{\vk u})\right\} \\
& = e^{<\vk\lambda_I(\tau),\vk x_I>}e^{<\tilde{\vk c}_I(\tau)/u,\vk x_I>}e^{<\tilde{\vk c}_J(\tau),\vk x_J>}e^{-\frac{1}{2}(\vk x/\bar{\vk u})^\top\Sigma^{-1}(\tau)(\vk x/\bar{\vk u})},
\end{align*}
where $\tilde{\vk c}(\tau) := \vk c^\top \Sigma^{-1}(\tau)$. From Lemma~\ref{lem:tau}(iii) we know that there exists $\eta>0$ such that $\Sigma^{-1}(\tau)-\eta \mathcal I_d$ is positive definite for all $\tau<T$, thus
\begin{align*}
e^{-\frac{1}{2}(\vk x/\bar{\vk u})^\top\Sigma^{-1}(\tau)(\vk x/\bar{\vk u})} \leq e^{-\frac{1}{2}(\vk x/\bar{\vk u})^\top\left(\Sigma^{-1}(\tau)-\eta I_d\right)(\vk x/\bar{\vk u})}e^{-\frac{\eta}{2}||\vk x/\bar{\vk u}||^2} \leq e^{-\frac{\eta}{2}\|\vk x_J\|^2}.
\end{align*}
Furthermore, due to the continuity of $v_i(\cdot)$ and $\lambda_i(\cdot)$ (see Lemma \ref{lem:lipshtz_continuity}), for all $\ep>0$ there exists $\tau_0<T$ large enough such that
\begin{align*}
\lambda_i(T)(1-\ep) < \lambda_i(\tau) < \lambda_i(T)(1+\ep), \quad\text{and}\quad c_i^*(T)(1-\ep) \leq |c_i^*(\tau)| \leq c_i^*(T)(1+\ep)
\end{align*}
for all $i\in\{1,\ldots,d\}$ and $\tau\in[\tau_0,T]$. Moreoever, for all $\ep>0$ small enough and $\tau_0>\bar\tau$, where $\bar\tau$ is defined in \nelem{lem:tau} we also have $\lambda_i^{\ep, \vk x}(t) := \lambda_i(t)-{\rm sgn}(x_i)\ep>0$ for all $i\in I$, thus for every $\ep>0$ small enough there exists $\tau_0$ such that
\bqny{
\theta_u(\vk x) \leq e^{< \vk \lambda_I(T)(1+\ep\cdot\sgn{\vk x_I}),\vk x_I>}e^{<\vk c_J^*(T)(1 + \ep\cdot\sgn{\vk x_J}),\vk x_J>}e^{-\frac{\eta}{2}||\vk x_J||^2}=:\bar\theta(\vk x),
}
Now, let $\vk Z^{\vk x}_{u,\tau}(t)$, and $\hat{\vk Z}_{u,\tau}(t)$ be defined as in \eqref{eq:def_Z_u}. Since $\vk X = A \vk Z$, then
\begin{eqnarray}
\label{eq:h_equalities} h_{u,\tau}(L, \vk x)&=&\pk{\exists_{t\in[\tau-Lu^{-2},\tau]}: A\vk Z(t)-\vk c t>u\vk a \mid A\vk Z(\tau)=\vk w_{u,\tau}(\vk x)},\\
\nonumber &=&\pk{\exists_{t\in[0,L]}: A\vk Z(\tau-\tfrac{t}{u^2})-\vk c(\tau-\tfrac{t}{u^2})>u\vk a \mid A\vk Z(\tau)=\vk w_{u,\tau}(\vk x)},\\
\nonumber &=&\pk{\exists_{t\in[0,L]}: A\vk Z^{\vk x}_{u, \tau}(t) - \vk c \tau + \vk c t/u^2 >u\vk a} \\
\nonumber &=&\pk{\exists_{t\in[0,L]} \forall_{i\in \{1,\ldots,d\}}: u(A\hat{\vk Z}_{u,\tau}(t))_i + (\vk\mu_{u,\tau}(t,\vk x))_i > 0},
\end{eqnarray}
where, with defining $R_{u,\tau}(t) := {\rm diag}(\rho_i(\tau-\tfrac{t}{u^2}, \tau) / v_i(\tau))$ for breviety, we have
\begin{align*}
\vk \mu_{u,\tau}(t,\vk x) & := uA\,R_{u,\tau}(t)A^{-1}\vk w_{u,\tau}(\vk x) - u\vk c\tau + \vk c t/u - u^2 \vk a \\
& = uA\,R_{u,\tau}(t)A^{-1} (u \tilde{\vk a} + \vk c\tau - \vk x/\bar{\vk u}) - u\vk c\tau + \vk c t/u - u^2 \vk a \\
& = u^2 A\left(R_{u,\tau}(t) - \mathcal I_d\right)A^{-1} \cdot (\tilde{\vk a} +\vk c\tau/u -\vk x/(u\bar{\vk u})) \\
& \quad\quad + \vk c t/u + u^2(\tilde{\vk a}-\vk a) - u\vk x/\bar{\vk u}.
\end{align*}
Now, notice that
\begin{align}
\nonumber u^2 A\left(R_{u,\tau}(t) - \mathcal I_d\right)A^{-1} & = u^2 A\cdot {\rm diag}\left(\frac{\rho_i(\tau-\tfrac{t}{u^2}, \tau)-v_i(\tau)}{v_i(\tau)}\right)\cdot A^{-1} \\
\nonumber & = tA\cdot {\rm diag}\left(\frac{1}{2v_i(\tau)}\left[\frac{v_i(\tau-\tfrac{t}{u^2})-v_i(\tau)}{t/u^2} + \frac{ v_i(\tfrac{t}{u^2})}{t/u^2}\right]\right)\cdot A^{-1} \\
\label{RQ_convergence}& \to -\frac{t}{2}AQ(\tau)A^{-1}, \quad u\to \infty
\end{align}
for any fixed $\tau,T$, where $Q(\tau) := {\rm diag}(\dot v_i(\tau)/ v_i(\tau))$. Moreover, applying the mean value theorem yields
\begin{align*}
\inf_{s\in[\tau_0-Lu^{-2},T]} |\dot v_i(s)| \leq \left|\frac{v_i(\tau-\tfrac{t}{u^2})-v_i(\tau)}{t/u^2}\right| \leq \sup_{s\in[\tau_0-Lu^{-2},T]} |\dot v_i(s)|,
\end{align*}
so using the assumption {\bf B0}, for every $\ep>0$ there exists $\tau_0<T$ such that
\begin{align*}
-(1+\ep)\dot v_i(T) \leq \frac{v_i(\tau-\tfrac{t}{u^2})-v_i(\tau)}{t/u^2} \leq -(1-\ep)\dot v_i(T)
\end{align*}
for all $\tau\in[\tau_0,T]$, $t\in[0,L]$ and $u$ large enough. The bound above implies that for any $\ep>0$, we can find $\tau_0<T$ such that for all $i\in I$ we have
\begin{align*}
(\vk \mu_{u,\tau}(t,\vk x))_{i} \leq \left(-\tfrac{1}{2} (AQA^{-1} \tilde{\vk a})_i + \ep\right)t - x_i(1-\ep\cdot\sgn{x_i}),
\end{align*}
for all $t\in[\tau_0,T]$ and $u$ large enough, where $Q:= Q(T) = {\rm diag}(\dot v_i(T)/ v_i(T))$.  In the following define $W_{u,\tau} := \max_{i\in\{1,\ldots,d\}}\sup_{t\in[0,L]} u(A\hat{\vk Z}_{u,\tau}(t))_i$ and see that
\begin{align*}
h_{u,\tau}(L, \vk x) & \leq \pk{\exists_{t\in[0,L]} \forall_{i\in I}: W_{u,\tau} + \left(-\tfrac{1}{2} (AQA^{-1}\tilde{\vk a})_i  + \ep\right)t - x_i(1 - \ep\cdot\sgn{x_i}) > 0} \\
& \leq \pk{\exists_{t\in[0,L]} \forall_{i\in I}: \frac{-\tfrac{1}{2} (AQA^{-1} \tilde{\vk a})_i + \ep}{1-\ep\cdot\sgn{x_i}} \cdot t > x_i - \frac{W_{u,\tau}}{1-\ep}}\\
& \leq \sum_{k=0}^\infty \pk{\exists_{t\in[0,L]} \forall_{i\in I}: \frac{-\tfrac{1}{2} (AQA^{-1} \tilde{\vk a})_i + \ep}{1-\ep\cdot\sgn{x_i}} \cdot t > x_i - \frac{W_{u,\tau}}{1-\ep}; W_{u,\tau}\in(\ep k,\ep(k+1)]}\\
& \leq \sum_{k=0}^\infty \mathbb{I}\left\{\exists_{t\in[0,L]} \forall_{i\in I}: \frac{-\tfrac{1}{2} (AQA^{-1} \tilde{\vk a})_i + \ep}{1-\ep\cdot\sgn{x_i}} \cdot t > x_i - \frac{\ep(k+1)}{1-\ep}\right\}\pk{W_{u,\tau}>\ep k}
\end{align*}
Furthermore, due to Corollary \ref{cor:weak_convergence_to_0} and assumption {\bf BII}, we have
\begin{align*}
\pk{W_{u,\tau}>\ep k} \leq e^{-(\ep k)^2}
\end{align*}
for all $\tau\in[\tau_0,T]$ and $u$ large enough. Thus,
\begin{align*}
h_{u,\tau}(L,\vk x) \leq \sum_{k=0}^\infty \mathbb{I}\left\{\exists_{t\in[0,L]} \forall_{i\in I}: \frac{-\tfrac{1}{2} (AQA^{-1} \tilde{\vk a})_i + \ep}{1-\ep\cdot\sgn{x_i}} \cdot t > x_i - \frac{\ep(k+1)}{1-\ep}\right\}e^{-(\ep k)^2} := \bar h(L,\vk x).
\end{align*}
for all $u$ large enough. Furthermore, define
\begin{align*}
E_k(L) := \int_{\R^{|I|}} \mathbb{I}\left\{\exists_{t\in[0,L]} \forall_{i\in I}: \frac{-\tfrac{1}{2} (AQA^{-1} \tilde{\vk a})_i + \ep}{1-\ep\cdot\sgn{x_i}} \cdot t > x_i - \frac{\ep(k+1)}{1-\ep}\right\}e^{< \vk \lambda_I(T)(1+\ep\cdot\sgn{\vk x_I}),\vk x_I>}\td\vk x_I.
\end{align*}
Then
\begin{align*}
\int_{\R^d} \bar h(L,\vk x)\bar\theta(\vk x) & = \sum_{k=0}^\infty E_k(L)e^{-(\ep k)^2} \cdot \int_{\R^{|J|}}e^{<\vk c_J^*(T)(1 + \ep\cdot\sgn{\vk x_J}),\vk x_J>}e^{-\frac{\eta}{2}||\vk x_J||^2}\td\vk x_J,
\end{align*}
Now, the integral over $\R^{|J|}$ above is bounded for all $\ep$ small enough because it does not depend on $\tau$ and $u$. We now focus on the sum $\sum_{k=0}^\infty E_k(L)e^{-(\varepsilon k)^2}$.  Let $\vk\delta = (\delta_1,\ldots,\delta_d)\in\{-1,1\}^{|I|}$. For each $k\in\N$ we have
\begin{align*}
E_k(L) \leq \sum_{\vk \delta\in\{-1,1\}^{|I|}} \int_{\R^{|I|}} \mathbb{I}\left\{\exists_{t\in[0,L]} \forall_{i\in I}: \frac{-\tfrac{1}{2} (AQA^{-1} \tilde{\vk a})_i + \ep}{1-\ep\delta_i} \cdot t > x_i - \frac{\ep(k+1)}{1-\ep}\right\}e^{< \vk \lambda_I(T)(1+\ep\delta_i),\vk x_I>}\td\vk x_I.
\end{align*}
After applying substitution $x_i := \lambda_i(1+\ep\delta_i)\left[x_i - \frac{\ep(k+1)}{1-\ep}\right]$, each term of the sum above is bounded from above by $\mathcal C(L ; \ep, \vk\delta)\cdot e^{g_i(\ep)}$, where
\begin{align*}
\mathcal C(L ; \ep, \vk\delta) := \frac{1}{\prod_{i\in I} \lambda_i(1+\ep\delta_i)}\int_{\R^{|I|}} \mathbb{I}\left\{\exists_{t\in[0,L]} \forall_{i\in I}: f_i(\ep,\delta_i) t > x_i\right\}e^{\sum_{i\in I} x_i}\td\vk x_I,
\end{align*}
with $f_i(\ep,\delta_i)$ and $g_i(\ep)$ defined below
\begin{align*}
f_i(\ep,\delta_i) := \frac{\left(-\tfrac{1}{2} (AQA^{-1} \tilde{\vk a})_i + \ep\right)\lambda_i(1+\ep\delta_i)}{1-\ep\delta_i}, \quad g_i(\ep) := \frac{\lambda_i(1+\ep)\ep(k+1)}{1-\ep}.
\end{align*}
It is straightforward to see that $f_i(\ep,\delta_i) \to f_i := (-A\frac{\dot{\vk v}(T)}{2\vk v(T)}A^{-1} \vk a)_i\lambda_i$, as $\ep\to0$ for any $\vk \delta$ and similarly $g_i(\ep)\to0$. Therefore,
\begin{align*}
\mathcal C(L ; \ep, \vk\delta) \to \mathcal C(L), \quad \ep\to0,
\end{align*}
where
\begin{align}
\label{def:CL}\mathcal C(L) & := \prod_{i\in I}\lambda_i\int_{\R^{|I|}}\mathbb{I}\left\{\exists_{t\in[0,L]}:-\tfrac{1}{2}(AQA^{-1} \tilde{\vk a})_I-\vk x_I>\vk 0_{I}\right\}e^{\sum\limits_{i\in I}\lambda_i x_i}\td\vk x_I\\
\nonumber & = \int_{\R^{|I|}}\mathbb{I}\left\{\exists_{t\in[0,L]}\forall_{i\in I}:-\tfrac{1}{2}\lambda_i \cdot (AQA^{-1} \tilde{\vk a})_i  -x_i>0\right\}e^{\sum\limits_{i\in I} x_i}\td\vk x.
\end{align}
Now, since $\Sigma^{-1} = A^{-\top}~{\rm diag}(v_i(T))A^{-1}$ and $\vk \lambda = \Sigma^{-1}\tilde a$, then
\begin{align*}
\sum_{i\in I}\lambda_i \cdot (AQA^{-1} \tilde{\vk a})_i & = \sprod{AQA^{-1} \tilde{\vk a},\vk\lambda} = \sprod{AQA^{-1} \tilde{\vk a},\vk\lambda} \\
& = \vk\lambda^\top~{\rm diag}\left(\frac{\dot{v_i}(T)}{v_i(T)}\right)A^{\top}A^{-\top}{\rm diag}\left(\frac{1}{v_i(T)}\right)A^{-1}\tilde{\vk a} \\
& = -\dot D(T),
\end{align*}
where in the last line we used Lemma~\ref{prop:D}. Applying Lemma \ref{lem:int_repr} with $\vk f = (f_i)$, $f_i = -\tfrac{1}{2}\sum_{i\in I}\lambda_i \cdot (AQA^{-1} \tilde{\vk a})_i$ and the fact that, $\sum_i f_i = \tfrac{1}{2}\dot D(T)<0$, cf. Lemma \ref{prop:D}, we conclude that $\mathcal C(L) \to \mathcal C$, as $L\to\infty$, with $\mathcal C$ defined in \eqref{def:C}.

Now, since $\mathcal C(L)\to \mathcal C$, then there must exist some $c_1>0$ such that for all $\ep$ small enough and all $\vk \delta\in\{-1,1\}^{|I|}$, we have $\mathcal C(L;\ep,\vk \delta) \leq (1+c_1) 2^{d+1}\mathcal C$. Finally, notice that for each $\ep>0$,
\begin{align*}
\sum_{k=0}^\infty \exp\left\{\frac{\lambda_i(1+\ep)\ep(k+1)}{1-\ep}\right\}e^{-(\ep k)^2} < \infty.
\end{align*}
These observations combined give us that there exists a constant $c_2>0$ such that
\begin{align*}
\int_{\R^d} \bar h(L,\vk x)\bar\theta(\vk x) < c_2 \cdot \mathcal C,
\end{align*}
for all $u$ large enough. This completes the proof.
\end{proof}

In the following, for any $\tau\in(0,T]$, $L>0$ and $u>\sqrt{L/\tau}$ let
\begin{equation}\label{def:Mtau}
M_\tau(u,L) := \pk{\exists_{t\in[\tau-Lu^{-2},\tau]}:\vk X(t)-\vk c t >u\vk a}.
\end{equation}

\begin{proof}[Proof of Lemma \ref{lem:A}]
For any $T>0$, with $M_T(u,L)$ defined in \eqref{def:Mtau}, we have
\bqny{
M_T(u,L)=
\int_{\R^d}\pk{\exists_{t\in[T-Lu^{-2},T]}:\vk X(t)-\vk c t>u\vk a \mid \vk X(T)=\vk x}\varphi_T(\vk x)\td\vk x,
}
where $\varphi_T$ be the pdf of $\vk X(T)$. After applying substitution $\vk w_{u,T}(\vk x)$; see \eqref{def:w} we obtain
\begin{align}
\nonumber M_T(u,L) & = u^{-|I_T|}\int_{\R^d}\pk{\exists_{t\in[T-Lu^{-2},T]}:\vk X(t)-\vk c t>u\vk a \mid \vk X(T)=\vk w_{u,T}(\vk x)}\varphi(\vk w_{u,T}(\vk x))\td\vk x \\
\label{eq:M_T}& = u^{-|I_T|}\int_{\R^d}h_{u,T}(L,\vk x)\varphi(\vk w_{u,T}(\vk x))\td\vk x.
\end{align}
Now, let $\theta_{u,T}(\vk x) := \varphi_T(\vk w_{u,T}(\vk x))/\varphi_T(u\tilde{\vk a}(T) +\vk cT)$. Then
\begin{align*}
\theta_{u,T}(\vk x) & = \exp\left\{u\tilde{\vk a}^\top(T)\Sigma^{-1}(T)(\vk x/\bar{\vk u}) + T\vk c^\top\Sigma^{-1}(T)(\vk x/\bar{\vk u}) - \frac{1}{2}(\vk x/\bar{\vk u})^\top \Sigma^{-1}(T)(\vk x/\bar{\vk u})\right\},
\end{align*}
with the three terms under the exponent exhibit the following behavior:
\bqny{
& &e^{-\frac{1}{2}(\vk x/\bar{\vk u})^\top \Sigma^{-1}(T)(\vk x/\bar{\vk u})}\to e^{-\frac{1}{2}\vk x_J^\top(\Sigma^{-1}(T))_{JJ}\vk x_J}, \quad u\to\infty\\
& &e^{T\vk c^\top\Sigma^{-1}(T)(\vk x/\bar{\vk u})}\to e^{T \sprod{(\vk c^\top \Sigma^{-1}(T))_J, \vk x_J}}, \quad u\to\infty\\
& &e^{\tilde{\vk a}^\top(T)\Sigma^{-1}(T)(u\vk x/\bar{\vk u})}=e^{\vk \lambda^\top(T)(u\vk x/\bar u)}= e^{\sprod{\vk \lambda_I(T), \vk x_I}}.
}
Letting $\tilde{\vk c} := Tc^\top\Sigma^{-1}(T)$ we obtain
\bqny{
\theta(\vk x):=\limit{u} \theta_{u}(\vk x) = e^{\sprod{\vk \lambda_I(T), \vk x_I}} \cdot e^{-\frac{1}{2}\vk x_J^\top(\Sigma^{-1}(T))_{JJ}\vk x_J}e^{\sprod{\tilde{\vk c}_J, \vk x_J}}.
}
Let $\hat{\vk Z}_{u, \tau}(t)$ be defined as in \eqref{eq:def_Z_u}. Repeating steps from the proof of Lemma \ref{lem:dominant}, cf. Eq.~\eqref{eq:h_equalities} and below, we find that
\begin{align*}
h_{u,T}(L, \vk x) = \pk{\exists_{t\in[0,L]}: uA\hat{\vk Z}_{u,\tau}(t) + \vk\mu_{u,T}(t,\vk x) > 0},
\end{align*}
where,
\begin{align*}
\vk\mu_{u,T}(t,\vk x) & = u^2 A\left(R_{u,\tau}(t) - \mathcal I_d\right)A^{-1} \cdot (\tilde{\vk a} +\vk c\tau/u - x/(u\bar{\vk u})) + \vk c t/u + u^2(\tilde{\vk a}-\vk a) - u\vk x/\bar{\vk u}.
\end{align*}
with $R_{u,\tau}(t) := {\rm diag}(\rho_i(\tau-\tfrac{t}{u^2}, \tau) / v_i(\tau))$. Let $U_T := \{i\in J_T: \tilde a_i(T) = a_i\}$ be the subset of $J_T$. Repeating steps from the proof of Lemma \ref{lem:dominant}, cf. Eq.~\eqref{RQ_convergence} and below, we obtain
\begin{align*}
(\vk\mu_{u,T}(t,\vk x))_i \to \mu(T,\vk x) := \begin{cases}
-\tfrac{1}{2}(AQA^{-1} \tilde{\vk a})_i - x_i, & i\in I\\
{\rm sgn}(-x_i) \cdot \infty, & i\in U \\
\infty, & i\in J\setminus U
\end{cases}
\end{align*}
as $u\to\infty$, where $Q:={\rm diag}(\dot v_i(T)/ v_i(T))$ and $U$ was defined in \eqref{keyWW}. Now, see that
\begin{eqnarray*}
\lefteqn{\pk{\exists_{t\in[0,L]} \forall_{i\in \{1,\ldots,d\}}: (\vk\mu_{u,T}(t,\vk x))_i > 0} \leq h_{u,T}(L, \vk x)} \\
&&\leq \pk{\exists_{t\in[0,L]} \forall_{i\in \{1,\ldots,d\}}: \sup_{t\in[0,L]} (uA\hat{\vk Z}_{u,\tau}(t))_i + (\vk\mu_{u,T}(t,\vk x))_i > 0}.
\end{eqnarray*}
Corollary \ref{cor:weak_convergence_to_0} implies that $\sup_{t\in[0,L]} (uA\hat{\vk Z}_{u,\tau}(t))_i \to 0$, as $u\to\infty$, so for every $\vk x\in\R$ we have
\begin{align*}
\lim_{u\to\infty} h_{u,T}(L, \vk x) = \mathbb{I}\left\{\exists_{t\in[0,L]}:-\tfrac{1}{2}(AQA^{-1} \tilde{\vk a})_I-\vk x_I>\vk 0_{I}\right\} \cdot \mathbb{I}\{\vk x_{U}<\vk 0_{U}\} =: h(L,\vk x)
\end{align*}
Thanks to Lemma \ref{lem:dominant}, we may apply Lebesgue's dominated convergence theorem and obtain,
as $u\to\infty$,
\begin{align*}
\int_{\R^d}h_{u,T}(L,\vk x)\theta_{u,T}(\vk x)\td\vk x & \to \int_{\R^d}h(L,\vk x)\theta(\vk x)\td\vk x \\
& = \mathcal C(L) \cdot \frac{1}{\prod_{i\in I} \lambda_i(T)}\int_{\R^{|J|}} e^{-\frac{1}{2}(\vk x_J-\tilde{\vk c})^\top(\Sigma^{-1}(T))_{JJ}(\vk x_J-\tilde{\vk c})}e^{\frac{1}{2}\tilde{\vk c}^\top(\Sigma^{-1}(T))_{JJ}\tilde{\vk c}} \mathbb{I}_{\{\vk x_U<\vk 0_U\}}\td\vk x_J,
\end{align*}
with $\mathcal C(L)$ defined in \eqref{def:CL}. Finally, using Lemma~\ref{prop:proj1} yields
\begin{align*}
\frac{M_T(u,L)}{\pk{\vk X(T)-\vk c T>u\vk a}} \to \mathcal C(L),
\end{align*}
as $u\to\infty$. Repeating the reasoning from the proof of Lemma \ref{lem:dominant}, we conclude that $\mathcal C(L) \to \mathcal C$, as $L\to\infty$, with $\mathcal C$ defined in \eqref{def:C}.
\end{proof}

\begin{proof}[Proof of Lemma \ref{lem:gap}]
Define a sequence $\tau_k := T-kLu^{-2}$ and a constant $K(u)=[\ln^2(u)+1]$. Then
\bqny{
\frac{\pk{\exists_{t\in[T-Lu^{-2}\ln^2(u),T-Lu^{-2}]}\vk X(t)-\vk ct>\vk au}}{\pk{\vk X(T)-\vk cT>u\vk a}}\leq \sum_{k=1}^{K(u)}\frac{M_{\tau_k}(u,L)}{\pk{\vk X(T)-\vk cT>u\vk a}},
}
with $M_\tau(u,L)$ defined in \eqref{def:Mtau}. Similarly to \eqref{eq:M_T}, for any $\tau>0$ we have
\begin{align*}
M_\tau(u,L) = u^{-|I_\tau|}\int_{\R^d}h_{u,\tau}(L,\vk x)\varphi_\tau(\vk w)\td\vk x,
\end{align*}
with $h_{u,\tau}$ defined as in \eqref{def:h,theta}. Let $\tau_0$ be as in Lemma~\ref{lem:dominant}. Then $\tau_k \in[\tau_0,T]$ for all $k\in\{1,\ldots,K(u)\}$ we may apply Lemma~\ref{lem:dominant} and obtain
\begin{align*}
M_{\tau_k}(u,L) \leq u^{-|I_T|} \varphi_{\tau_k}(u\tilde{\vk a}(\tau_k)+\vk c \tau_k) \cdot \int_{\R^d} H(L, \vk x)\td\vk x
\end{align*}
for all $u$ large enough. Furthermore, according to Lemma~\ref{lem:dominant}, there exists a constant $C_1$ such that $\int_{\R^d} H(L, \vk x)\td\vk x \leq C_1$ for all $u$ large enough. Moreover, according to Lemma \ref{prop:proj1}, there exist $C_2>0$, such that $\pk{\vk X(T)-\vk cT>u\vk a}\geq C_2u^{-|I_T|}\varphi_T(\tilde{\vk a}u+\vk c T)$ for all $u$ large enough. We thus have
\bqny{
\frac{\pk{\exists_{t\in[T-Lu^{-2}\ln^2(u),T-Lu^{-2}]}\vk X(t)-\vk ct>\vk au}}{\pk{\vk X(T)-\vk cT>u\vk a}}\leq \frac{C_1}{C_2}\sum_{k=1}^{K(u)}\frac{\varphi_{\tau_k}(u\tilde{\vk a}(\tau_k)+\vk c \tau_k)}{\varphi_T(u\tilde{\vk a}(T)+\vk c T)}.
}
Consider one of the terms in the sum above. We have
\begin{align*}
\frac{\varphi_{\tau_k}(u\tilde{\vk a}(\tau_k)+\vk c \tau_k)}{\varphi_T(u\tilde{\vk a}(T)+\vk c T)} = \sqrt{\frac{|\Sigma(T)|}{|\Sigma(\tau_k)|}} \cdot e^{-\frac{u^2}{2}(D(\tau_k)-D(T))}e^{-u\sprod{\vk\lambda(\tau_k)-\vk\lambda(T), \vk c}}e^{-\frac{1}{2}c^\top(\Sigma^{-1}(\tau_k)-\Sigma^{-1}(T))c}.
\end{align*}
Now, using the fact that $D(t)$ has a negative derivative at $t=T$ (cf. Lemma \ref{prop:D}), and the fact that $\tau_{K(u)}\to T$ we know that for $u$ large enough $u^2(D(\tau_k)-D(T))\geq |\dot D(T)|kL/2$. Using Lipshitz continuity of $\lambda_i$ we also find that there exists a constant $C_1>0$ such that $|\sprod{\lambda_i(\tau_k)-\lambda_i(T), \vk c}|\leq C_1kL$. Both these observations combined imply that there exist some constants $C_3$, $\beta>0$ such that
\begin{align*}
\frac{\varphi_{\tau_k}(u\tilde{\vk a}(\tau_k)+\vk c \tau_k)}{\varphi_T(u\tilde{\vk a}(T)+\vk c T)} \leq C_3 e^{-\beta kL}
\end{align*}
for all $u$ large enough. Finally, we have
\begin{align*}
\sum_{k=1}^{K(u)}\frac{M_{\tau_k}(u,L)}{\pk{\vk X(T)-\vk cT>u\vk a}} & \leq \frac{C_1C_3}{C_2} \sum_{k=1}^{\infty} e^{-\beta kL} \leq \frac{C_1C_3}{C_2(1-e^{-\beta L})} \cdot e^{-\beta L},
\end{align*}
which completes the proof.
\end{proof}

\section*{Acknowledgements}

We are in debt to Enkelejd Hashorva for stimulating discussions and suggestions that
substantially improved the manuscript and to
Pavel Ievlev for pointing out errors in earlier version of this manuscript.
Krzysztof Bisewski's and Nikolai Kriukov's research was funded by SNSF Grant 200021-196888.
Krzysztof D{\c{e}}bicki was partially supported by NCN Grant No 2018/31/B/ST1/00370
(2019-2022).

\bibliographystyle{ieeetr}

\bibliography{SGR}

\def\polhk#1{\setbox0=\hbox{#1}{\ooalign{\hidewidth
  \lower1.5ex\hbox{`}\hidewidth\crcr\unhbox0}}}
\begin{thebibliography}{10}

\bibitem{MR789369}
S.~M. Berman, ``An asymptotic formula for the distribution of the maximum of a
  {G}aussian process with stationary increments,'' {\em J. Appl. Probab.},
  vol.~22, no.~2, pp.~454--460, 1985.

\bibitem{DeR02}
K.~D{\c{e}}bicki and T.~Rolski, ``A note on transient {G}aussian fluid
  models,'' {\em Queueing Systems}, vol.~41, no.~4, pp.~321--342, 2002.

\bibitem{DHM19}
K.~D\c{e}bicki, E.~Hashorva, and Z.~Michna, ``Simultaneous ruin probability for
  two-dimensional brownian risk model,'' {\em Journal of Applied Probability},
  vol.~57, no.~2, pp.~597--612, 2020.

\bibitem{KoW20}
D.~Korshunov and L.~Wang, ``Tail asymptotics for {S}hepp-statistics of
  {B}rownian motion in ${R}^d$,'' {\em Extremes}, vol.~23, no.~1, pp.~35--54,
  2020.

\bibitem{DHK21}
K.~D{\polhk{e}}bicki, E.~Hashorva, and N.~Kriukov, ``Pandemic-type failures in
  multivariate {B}rownian risk models,'' {\em arXiv:2008.07480}, 2021.

\bibitem{Ji18}
L.~Ji and S.~Robert, ``Ruin problem of a two-dimensional fractional {B}rownian
  motion risk process,'' {\em Stochastic Models}, vol.~34, no.~1, pp.~73--97,
  2018.

\bibitem{FPP08}
F.~Avram, Z.~Palmowski, and M.~R. Pistorius, ``Exit problem of a
  two-dimensional risk process from the quadrant: exact and asymptotic
  results,'' {\em The Annals of Applied Probability}, vol.~18, no.~6,
  pp.~2421--2449, 2008.

\bibitem{HuJ13}
Z.~Hu and B.~Jiang, ``On joint ruin probabilities of a two-dimensional risk
  model with constant interest rate,'' {\em Journal of Applied Probability},
  vol.~50, no.~2, pp.~309--322, 2013.

\bibitem{FPR17}
S.~Foss, D.~Korshunov, Z.~Palmowski, and T.~Rolski, ``Two-dimensional ruin
  probability for subexponential claim size,'' {\em Probability and
  Mathematical Statistics}, vol.~37, no.~2, pp.~319--335, 2017.

\bibitem{Mic20}
Z.~Michna, ``Ruin probabilities for two collaborating insurance companies,''
  {\em Probability and Mathematical Statistics}, vol.~40, no.~2, pp.~369--386,
  2020.

\bibitem{DHW20}
K.~D{\c{e}}bicki, E.~Hashorva, and L.~Wang, ``Extremes of vector-valued
  {G}aussian processes,'' {\em Stochastic Processes and their Applications},
  vol.~130, no.~9, pp.~5802--5837, 2020.

\bibitem{HHJT15}
K.~D{\c{e}}bicki, E.~Hashorva, L.~Ji, and K.~Tabi{\'s}, ``Extremes of
  vector-valued {G}aussian processes: Exact asymptotics,'' {\em Stochastic
  Processes and their Applications}, vol.~125, no.~11, pp.~4039--4065, 2015.

\bibitem{Mic98}
Z.~Michna, ``Self-similar processes in collective risk theory,'' {\em Journal
  of Applied Mathematics and Stochastic Analysis}, vol.~11, no.~4,
  pp.~429--448, 1998.

\bibitem{HuP99}
J.~H{\"u}sler and V.~Piterbarg, ``Extremes of a certain class of {G}aussian
  processes,'' {\em Stochastic Processes and their Applications}, vol.~83,
  no.~2, pp.~257--271, 1999.

\bibitem{HuP04}
J.~H{\"u}sler and V.~Piterbarg, ``On the ruin probability for physical
  fractional {B}rownian motion,'' {\em Stochastic Processes and their
  Applications}, vol.~113, no.~2, pp.~315--332, 2004.

\bibitem{Deb02}
K.~D{\c{e}}bicki, ``Ruin probability for {G}aussian integrated processes,''
  {\em Stochastic Processes and their Applications}, vol.~98, no.~1,
  pp.~151--174, 2002.

\bibitem{debicki2020extremes}
K.~D{\c{e}}bicki, E.~Hashorva, and L.~Wang, ``Extremes of vector-valued
  {G}aussian processes,'' {\em Stochastic Processes and their Applications},
  vol.~130, no.~9, pp.~5802--5837, 2020.

\bibitem{Has05}
E.~Hashorva, ``Asymptotics and bounds for multivariate {G}aussian tails,'' {\em
  Journal of theoretical probability}, vol.~18, no.~1, pp.~79--97, 2005.

\bibitem{Has19}
E.~Hashorva, ``Approximation of some multivariate risk measures for {G}aussian
  risks,'' {\em Journal of Multivariate Analysis}, vol.~169, pp.~330--340,
  2019.

\bibitem{MR800188}
Y.~Gordon, ``Some inequalities for {G}aussian processes and applications,''
  {\em Israel J. Math.}, vol.~50, no.~4, pp.~265--289, 1985.

\bibitem{MR1088478}
R.~J. Adler, {\em An introduction to continuity, extrema, and related topics
  for general {G}aussian processes}, vol.~12 of {\em Institute of Mathematical
  Statistics Lecture Notes---Monograph Series}.
\newblock Institute of Mathematical Statistics, Hayward, CA, 1990.

\bibitem{Pit96}
V.~I. Piterbarg, {\em Asymptotic methods in the theory of {G}aussian processes
  and fields}, vol.~148 of {\em Translations of Mathematical Monographs}.
\newblock Providence, RI: American Mathematical Society, 1996.
\newblock Translated from the Russian by V.V. Piterbarg, revised by the author.

\bibitem{Hag79}
W.~W. Hager, ``Lipschitz continuity for constrained processes,'' {\em SIAM J.
  Control Optim.}, vol.~17, no.~3, pp.~321--338, 1979.

\end{thebibliography}

\end{document}